\documentclass[preprint,12pt]{elsarticle}

\usepackage{amssymb,amsthm,amsmath,hyperref,txfonts}

\usepackage{hyperref}
\usepackage{graphicx,float}

\newtheorem{theorem}{Theorem}[section]
\newtheorem{corollary}[theorem]{Corollary}
\newtheorem{lemma}{Lemma}[section]
\newtheorem{proposition}{Proposition}[section]

\newtheorem{remark}{Remark}[section]
\numberwithin{equation}{section}

\def\a{\alpha}

\def\f{\frac}

\def\b{\bar}

\usepackage{tikz}

\newcommand{\dd}{{\rm d}}

\newcommand{\FU}{\mathbf{U}}

\newcommand{\CD}{\mathcal {D}}
\newcommand{\CE}{\mathcal{E}}

\newcommand{\CL}{\mathcal{L}}

\newcommand{\CM}{\mathcal{M}}

\newcommand{\CV}{\mathcal{V}}

\newcommand{\CU}{\mathcal{U}}
\newcommand{\CP}{\mathcal{P}}

\newcommand{\cw}{\mathcal{W}}

\newcommand{\cu}{\mathfrak{u}}
\newcommand{\cv}{\mathfrak{v}}

\newcommand{\ur}{\underline{\rho}}
\newcommand{\uu}{\underline{u}}
\newcommand{\uv}{\underline{v}}

\newcommand{\pa}{\partial}

\newcommand{\eps}{\epsilon}

\newcommand{\fq}{\mathfrak{Q}}
\newcommand{\fs}{\mathfrak{S}}

\newcommand{\eqdef}{\overset{\mbox{\tiny{def}}}{=}}
\newcommand{\dv}{\text{div}_k}

\allowdisplaybreaks

\journal{Elsevier journal }

\begin{document}
	
	\begin{frontmatter}
		
		
		
		\title{ Long-time instability of planar Poiseuille-type flow in compressible fluid}
		
		
		\author[Y]{Andrew Yang}
		\address[Y]{Department of  Mathematics, City University of Hong Kong, Hong Kong\\
		andryang-c@my.cityu.edu.hk}
	    \author[Z]{Zhu Zhang}
	    \address[Z]{Department of Applied Mathematics, The Hong Kong Polytechnic University, Hong Kong\\
	    zhuama.zhang@polyu.edu.hk}

		\begin{abstract}
		It is well-known that at the high Reynolds number, the linearized Navier-Stokes equations around the inviscid stable shear profile admit growing mode solutions due to the destabilizing effect of the viscosity. This phenomenon, called Tollmien-Schlichting instability, has been rigorously justified by Grenier-Guo-Nguyen [Adv. Math. 292 (2016); Duke J. Math. 165 (2016)] for Poiseuille flows and  boundary layers in the incompressible fluid. To reveal this intrinsic instability mechanism in the compressible setting, in this paper, we study the long-time instability of the Poiseuille flow in a channel.
	Note that this instability arises in a low-frequency regime instead of a high-frequency regime  for the
	Prandtl boundary layer. The proof is based on the quasi-compressible-Stokes iteration introduced by Yang-Zhang in \cite{YZ} and subtle analysis of the dispersion relation for the instability. Note that we do not require symmetric conditions on the background shear flow or perturbations.
		\end{abstract}		
	\end{frontmatter}
\section{Introduction}

Stability/instability properties of laminar flows at the high Reynolds number is the fundamental problem in {\it Hydrodynamic Stability Theory}. Significant progress was made by  pioneering works of Kelvin \cite{K}, Orr \cite{Orr}, Rayleigh \cite{R}, and Heisenberg \cite{H}, among many others. The mathematical theory in this area has been rapidly developed in this decade.

For inviscid flow when the Reynolds number is infinity, the main stabilizing mechanism is the {inviscid damping.}  That is, due to the mixing between vorticity and the background flow, the velocity field  asymptotically converges to a laminar flow that is closed to its background.  This phenomenon has been rigorously justified for Couette flow in linear/nonlinear settings, and for perturbations in high-regularity spaces, cf. \cite{BM,J,Z1,Z}.  On the other hand, if the perturbation is less smooth, some instabilities arise that prevent the solution from converging to the shear flow. In this direction, Lin-Zeng \cite{LZ} constructed the time-periodic solution close to the Couette flow in $H^{\f32-}$-space. Nonlinear instability in Gevrey $2^-$ space of Couette flow was shown by Deng-Masmoudi \cite{DM}. For shear flows other than the Couette, the dynamics is more complicated even in the linear level \cite{BCV,GNRS,J,LZW,WZZ1,WZZ}, and the stability in  full nonlinear setting is shown recently by Ionescu-Jia \cite{IJ}
and Masmoudi-Zhao \cite{MZ}. We also refer to \cite{G1,LMZ} for nonlinear instability results of generic shear flows.

In the high Reynolds number regime, the role of viscosity is mixed. On the one hand, the viscosity can help stabilize the shear flow with the aid of an enhanced dissipation mechanism. Such a mechanism can also be used to characterize the critical size of perturbations for the nonlinear stability of the  corresponding shear flows. We refer to the list of references \cite{BGM1,BGM2,BMV,BH,CLWZ,CWZ,CEW,DL,GNRS,MZ1,MZ,Romanov,WZZ} for the research in this direction. On the other hand, small viscosity may destabilize the flow. This phenomenon, called Tollmien-Schlichting instability, cf. \cite{DR,SG,W}, typically occurs in the laminar flow with boundary layers. Precisely, for the linearized incompressible Navier-Stokes equations  around a 
Prandtl boundary layer, at high tangential frequency, there exists an unstable eigenmode corresponding to Gevrey function space with index $\frac 32$.
Substantial progress on this topic was made by Grenier-Guo-Nguyen \cite{GGN1}
in which they  justified this phenomenon with mathematical rigor. In addition, the validity of Prandtl expansion with critical Gevrey index $\frac 32$ was obtained by G\'erard-Varet-Maekawa-Masmoudi in \cite{GMM1}. For  plane Poiseuille flow, this mechanism induces a long-wave type of instability, cf. \cite{GGN2}.



One of the powerful analytic tools in above studies for incompressible flow is the Orr-Sommerfeld equation, which describes the time evolution of the
Fourier mode behavior of the stream function.  In the absence of viscosity, the Orr-Sommerfeld equation is the Rayleigh equation for inviscid flow. 
Thus, it can be regarded as a singular perturbation of the Rayleigh equation at high Reynolds numbers.
The derivation of this famous equation relies on the stream function representation of the velocity field so that it is effective in the two-dimensional incompressible setting. However, for compressible flow, eigenvalue problems for stability can not be reduced to a simple single equation like the Orr-Sommerfeld. Therefore, despite many works from the physical perspectives \cite{B,Dk,LR,L,MK,R1},  mathematical results on hydrodynamic stability in compressible fluids are very few. The works that we are aware of in this direction are \cite{ADM,ZZZ} for the Couette flow in two and three-dimensional spaces (without boundary) respectively. See also \cite{KN,KN2,MP} for stability of Poiseuille flow in different regimes.

In order to
capture the Orr-Sommerfeld structure in the compressible setting, in \cite{YZ}, the authors introduced a quasi-compressible system so that the Orr-Sommerfeld type equation can be derived and used to approximate the compressible Navier-Stokes for isentropic flow. This newly derived formulation allows one to use the ``incompressible techniques'' to treat problems for compressible flows, and it 
has been proved helpful in the subsonic regime through the quasi-compressible and Stokes iteration to investigate both the Tollmien-Schlichting instability and validity of Prandtl ansatz in suitable function spaces in various physical situations. We also refer to the recent work \cite{MWWZ} for the whole subsonic regime.

As a further exploration of the compressible flow, particularly in a channel, we will study the long-time instability phenomena of the planar Poiseuille flow in this paper. Note that although this kind of instability phenomena has been studied in \cite{GGN2} for Poiseuille flows in the incompressible fluid,  there is no corresponding work for compressible models. Motivated by this, we consider the two-dimensional  compressible Navier-Stokes equations for isentropic flow in a channel $\Omega=\left\{(x,y)\mid x\in \mathbb{R},y\in [-1,1]\right\}:$
\begin{equation}\label{1.1}
	\left\{\begin{aligned}
		&\pa_t\rho+\nabla\cdot(\rho \vec{U})=0,\\
		&\rho\pa_t\vec{U}+\rho\vec{U}\cdot\nabla \vec{U}+\nabla P(\rho)=\mu\eps\Delta \vec{U}+\lambda\eps\nabla(\nabla\cdot\vec{U})+\rho\vec{F},\\
		&\vec{U}|_{y=\pm 1}=\vec{0}.
	\end{aligned}\right.
\end{equation}
Here, $\rho$,  $\vec{U}=(u,v)$ and $P(\rho)$ represent the density, velocity field and pressure of the fluid respectively. The vector field $\vec{F}$ is a given external force. 
The constants $\mu>0, \lambda\geq 0$ are the  rescaled shear and bulk viscosity, and  $0<\eps\ll1$ is a small parameter  which is proportional to the reciprocal of the Reynolds number. Without loss of generality, the constant $\mu$ is set to 1 throughout the paper. 

As well known, the  planar Poiseuille flow is a typical steady solution profile to the Navier-Stokes equations when 
the tangential derivative of the pressure function is balanced by the normal diffusion of the velocity field. In the
compressible setting under consideration, we can fix a class of Poiseuille-type profile as 
\begin{align}\label{1.1.1}
	(\rho_s,\vec{\FU}_s)\eqdef(1,U_s(y),0),~y\in (-1,1),
\end{align}
which is a steady solution to \eqref{1.1} with a small constant force $\vec{F}=(-2U_s''\mu,0)$. Here we assume that
 $U_s$ satisfies
\begin{align}
	U_s(\pm 1)=0,~U_s'(-1)>0,~U_s'(1)<0,\text{ and }U_s''(y)<0, ~\text{ for }y\in (-1,1).\label{as}
\end{align}
A typical example of this class of profile is the plane Poiseuille flow $U_s(y)=1-y^2$.

To study the stability/instability of the Poiseuille flow, we consider a  small  perturbation of the Poiseuille profile
\begin{align}
	\rho=1+\tilde{\rho},~u=U_s(y)+\tilde{u},~v=\tilde{v}.\nonumber
\end{align} 
By linearization, we have the following system for $(\rho, \vec{U})$
\begin{equation}\label{1.3}
	\left\{
	\begin{aligned}
		&\pa_t\rho+U_s\pa_x\rho+\nabla\cdot \vec{U}=0,~t>0,~(x,y)\in \mathbb{R}\times{(-1,1)},\\
		&\pa_t\vec{U}+U_s\pa_x\vec{U}+m^{-2}\nabla \rho+v\pa_yU_s\vec{e}_1-{\eps}\Delta\vec{U}-\lambda\eps\nabla (\nabla\cdot \vec{U})-\rho\vec{F}=0,\\
		&\vec{U}|_{y=\pm 1}=\vec{0}.
	\end{aligned}
	\right.
\end{equation}
Here we have dropped tilde for the simplicity of notations. Note that in our setting,  $m=\frac{1}{\sqrt{P'(1)}}$ is the Mach number because the maximum magnitude of
the Poiseuille profile is $1$ and $\sqrt{P'(1)}$ is the background sound speed.

For stability analysis, we look for solution $(\rho,u,v)$ to the linearized compressible Navier-Stokes system \eqref{1.3} in the following form
\begin{align}\label{1.2}
	(\rho,u,v)=e^{ik(x-ct)}(\hat{\rho},\hat{u},\hat{v})(y),
\end{align}
where  $k>0$ is  frequency and $c\in \mathbb{C}$ is the spectral parameter. Substituting \eqref{1.2} into \eqref{1.3}, we obtain the following eigenvalue problem for the profile function $(\hat{\rho},\hat{u},\hat{v})$
\begin{equation}\label{1.4}
	\left\{
	\begin{aligned}
		&ik(U_s-c)\hat{\rho}+\dv(\hat{u},\hat{v})=0,~~y\in (-1,1),\\
		&-{\eps}\Delta_k \hat{u}-\lambda ik{\eps}\dv(\hat{u},\hat{v})+ik(U_s-c)\hat{u}+(ik m^{-2}+{\eps}U_s'')\hat{\rho}+\hat{v}U_s'=0,\\
		&-{\eps}\Delta_k \hat{v}-\lambda {\eps}\pa_y\dv(\hat{u},\hat{v})+ik(U_s-c)\hat{v}+m^{-2}\pa_y\hat{\rho}=0,
	\end{aligned}
	\right.
\end{equation}
with no-slip boundary conditions on two boundaries $y=\pm 1$
\begin{align}\label{1.4-1}
	\hat{u}|_{y=\pm 1}=\hat{v}|_{y=\pm 1}=0.
\end{align}
Here, in \eqref{1.4}, 
$\Delta_k =(\pa_y^2-k^2)$ and	
$\dv(u,v)=ik u+\pa_y v$  denote the Fourier transform of Laplacian and divergence operators respectively. For convenience, we denote by $\CL(\hat{\rho},\hat{u},\hat{v})$ the linear operator \eqref{1.4}. 

The stability analysis relies on the solvability of the eigenvalue problem \eqref{1.4} with boundary conditions \eqref{1.4-1}. In fact, 
if for some complex number $c\in \mathbb{C}$ with $\text{Im}c>0$ and frequency $k\in \mathbb{R}$, the ODE system \eqref{1.4} with \eqref{1.4-1} has a non-zero solution, then the Poiseuille flow \eqref{1.1.1} is spectral unstable. Otherwise, it is spectral stable.

In this paper we concentrate on the case when the background profile is the planar shear flow that satisfies \eqref{as}, and the main result is stated in the following theorem. 
\begin{theorem}\label{thm1.1}
	Let Mach number $m\in\left(0,\frac{1}{\sqrt{3}}\right)$. For sufficiently small $0<\eps\ll 1$,  there exist  frequency $k\sim\epsilon^\frac17$ and wave number $c\in\mathbb{C}$ with $Imc\sim\epsilon^\frac27$, such that the linearized compressible Navier-Stokes system admits a solution in the form of
	$$(\rho,u,v)(t,x,y)=e^{ik(x-ct)}(\hat{\rho}.\hat{u},\hat{v})(y).$$
\end{theorem}
We have several remarks on the above result.

\begin{remark}
	Unlike the case of boundary layer profile \cite{GGN1,YZ}, for the planar Poiseuille flow, the 
	unstable mode localizes at low frequency $k\sim\epsilon^\frac17$  rather than in a high-frequency regime, and it grows exponentially in time like  $e^{t kImc }\sim e^{\epsilon^\frac37 t}$. Note that the slow growth is consistent with the linear inviscid stability for Poiseuille flow, see Lees-Lin's criterion \cite{LL}.
\end{remark}

\begin{remark}
	In the work \cite{KN} by Kagei-Nishida, the instability of planar Poiseuille flows is shown in the high Mach number regime, which is different from the regime studied in this paper. Even though  instabilities investigated in both papers occur at low frequencies, the mechanisms are different. In \cite{KN}, the unstable eigenmode bifurcates from the zero eigenvalue of hyperbolic part of linearized operator at zero frequency, 
	while in this paper, the instability arises from the interaction between  the inviscid perturbation and boundary layers.
\end{remark}

\begin{remark}
	The dispersion relation  is derived from full boundary conditions \eqref{1.4-1}, instead of building upon symmetry assumptions on background flows and perturbations as in \cite{GGN2}. Thus, the analysis can be used for more generic shear flows without symmetry.
\end{remark}
\begin{remark}
	The above theorem is proved in the subsonic regime, covering the incompressible case when the Mach number is $0$. Note that the subsonic condition is essential because the ellipticity of  compressible Orr-Sommerfeld equation does not hold when the Mach number $m\ge 1$. Hence, 
	how to establish mathematical theory for the cases in transonic and sonic regimes remains very challenging and unsolved.
\end{remark}

Now we briefly present the ideas and strategies for the proof of main theorem. To match the full boundary conditions in \eqref{1.4-1}, we need to construct four independent solutions to the eigenvalue problem \eqref{1.4}.  As in the case of boundary layer profiles \cite{GGN1,YZ} and the case of Poiseuille flow under incompressible perturbations \cite{GGN2}, the first step is to construct approximate solutions consisting of the inviscid and viscous parts. The inviscid component is constructed by using the Rayleigh approximate operator in the low-frequency regime,  usually called the slow mode. In constrast, the viscous component takes effect in a region close to the two boundaries $y=\pm 1$ to capture  interactions between the viscosity and no-slip boundary conditions.  

The next step is to construct four exact solutions to \eqref{1.4} near these approximations or equivalently to solve the remainder system \eqref{3.1}. The difficulty in analysis comes from the lack of stream function expression in \eqref{1.4}, so the resolvent problem \eqref{3.1.1} can not be  reduced to the Orr-Sommerfeld equation as in the incompressible case. To overcome this difficulty, we apply the quasi-compressible-Stokes iteration introduced in \cite{YZ} for studying  boundary layers to the  study of flows in the channel. Roughly speaking, we introduce the quasi-compressible operator, where the intrinsic effective stream function can be defined by considering the compressibility of fluids. The Orr-Sommerfeld type equation can be derived accordingly. On the other hand, to recover the loss of
regularity due to the quasi-compressible approximation, the Stokes operator is used to capture the elliptic
structure of the system. The convergence of  iteration scheme will be presented in Section 3.3. 
The proof is based on energy estimates, which is different from the Green function approach used in \cite{GGN2}. The concavity of  Poiseuille flow plays an important role in analysis.

Finally, in Section 4, we solve the dispersion relation for the instability of Poiseuille flow. We remark that the dispersion relation is more complicated than the one studied in \cite{YZ} for the instability of boundary layer, because, for latter case, one can choose the slow and fast modes decaying at infinity so that far-field conditions can be automatically satisfied. In this paper, to remove the  symmetry assumptions made in \cite{GGN2}, the dispersion relation is derived from full boundary conditions \eqref{1.4-1}: we look for the solution to eigenvalue problem \eqref{1.4} in the following form
\begin{align}
	(\rho,u,v)=(\rho_+^s,u_+^s,v_+^s)+\a_1(\rho_-^s,u_-^s,v_-^s)+\a_2(\rho_-^f,u_-^f,v_-^f)+\a_3(\rho_+^f,u_+^f,v_+^f),\nonumber
\end{align}
where $(\rho_\pm^s,u_\pm^s,v_\pm^s)$ are slow modes and $(\rho_\pm^f,u_\pm^f,v_\pm^f)$ are fast modes that involve boundary layer structures near $y=\pm1$ respectively, and constants $\a_1$, $\a_2$, $\a_3$ are chosen such that the following three boundary conditions are fulfilled: $v(-1)=v(1)=u(-1)=0$. 
In this way, the zero point problem for the $4\times 4$ matrix of dispersion relation is reduced into a single equation for  matching the last boundary condition $u(1)=0.$ Inspired by  \cite{DDGM,DG}, we solve this  by applying the Rouch\'e theorem.

The rest of the paper is organized as follows. In the next section, we will construct the approximate solutions consisting of the inviscid modes and boundary layers corresponding to the inviscid and viscous effects respectively. Then, in the third section, the remainder of these approximate solutions are estimated by studying the resolvent problem. The quasi-compressible and Stokes iteration is used there. The dispersion relation for the instability is solved in Section 4. Finally, we list several useful estimates in the Appendices.

{\it Regime of parameters:} For the sake of readers' convenience, we list the regimes of the parameters considered in this paper. Let $\eps$ be the viscosity coefficient which is suffciently small. We set the tangential frequency $k=T_0\eps^{\f17}$ where the constant $T_0\gg1 $ is sufficiently large and of $O(1)$ with respect to $\eps.$ We also define the rescaled frequency $n\eqdef k/\eps=T_0\eps^{-\f67}$. The wave speed $c=\text{Re}c+i\text{Im}c$ and satisfies
\begin{align}\label{c}
	\text{Re}c\approx T_0^2\eps^{\f27},\text{ and } \text{Im}c\approx T_0^{-\f32}\eps^{\f27}.
\end{align} 
The precise range of $c$  is given in \eqref{cd}.

{\it Notations:} Throughout the paper, $C$ denotes a positive constant independent of $\eps$ and may vary from line to line. $A\lesssim B$ and $A=O(1)B$ mean that there exists a positive constant $C$, not depending on $\eps$, such that $A\leq CB$. Similarly, $A\gtrsim B$ means $A\geq C B$ for some positive constant $C$. $A=O(1)\eps^{\infty}$ means $A\leq C_N\eps^{N}$ for any positive integer $N$. We use the notation $A\approx B$ if $A\lesssim B$ and $A\gtrsim B.$ We denote by $\|\cdot\|_{L^2}$ and $\|\cdot\|_{L^\infty}$ respectively the standard $L^2$ and $L^\infty$ norm in the interval $(-1,1).$ For any non-negative integer $s$ and $p\in [1,\infty]$,  the standard Sobolev space $W^{s,p}=\{f \in L^p(-1,1)\mid \pa_y^jf \in L^p(-1,1),~j=1,2,\cdots, s\}$
are used. In particular, $H^s\eqdef W^{s,2}$.

\section{Approximate solutions}

The approximate solutions  consisting of inviscid approximations and boundary layers will be given in the following two subsections. The estimates on errors generated by these approximations will be given in Section 2.3.

\subsection{Inviscid solutions}
We start from the following inviscid system corresponding to  \eqref{1.4}
\begin{equation}
	\left\{
	\begin{aligned}\label{2.1.1}
		&ik(U_s-c)\rho+\dv(u,v)=0,~~y\in (-1,1),\\
		&ik(U_s-c)u+m^{-2}ik\rho+vU_s'=0,\\
		&ik(U_s-c)v+m^{-2}\pa_y\rho=0.
	\end{aligned}\right.
\end{equation}
As in \cite{LL,YZ}, the solution $(\rho,u,v)$ to the inviscid system \eqref{2.1.1} can be represented by
using a new function $\varphi=\f{i}{k}v$. In fact,  the first equation in \eqref{2.1.1} gives
\begin{align}\label{2.1.2}
	u=\pa_y\varphi-(U_s-c)\rho.
\end{align}
Substituting this into the second equation of \eqref{2.1.1}, we can obtain the following relation between $\rho$ and $\varphi$:
\begin{align}
	-m^{-2}A(y)\rho=(U_s-c)\pa_y\varphi-\varphi U_s',\nonumber
\end{align}
where 
\begin{align}\label{2.1.4}
	A(y)\eqdef 1-m^2(U_s-c)^2.
\end{align}
Note that for the Poiseuille flow in the {\it subsonic} regime, that is, $m\in (0,1)$, $A(y)$ is invertible when the module $|c|$ is sufficiently small. Thus, we can represent $\rho$ in terms of $\varphi$ as follows:
\begin{align}\label{2.1.3}
	\rho=-m^2A^{-1}(y)\left[ (U_s-c)\pa_y\varphi-\varphi U_s'\right].
\end{align}
Plugging \eqref{2.1.3} into the third line of \eqref{2.1.1}, we derive the following Lees-Lin equation (Rayleigh's type) for $\varphi$, cf. \cite{LL,YZ}:
\begin{align}\label{2.1.5}
	\text{Ray}_{\text{CNS}}(\varphi)\eqdef\pa_y\left\{ A^{-1}\left[ (U_s-c)\pa_y\varphi-U_s'\varphi \right]\right\}-k^2(U_s-c)\varphi=0,~~y\in(-1,1).
\end{align}

Compared with the classical Lees-Lin equation around the Prandtl boundary layer profile studied in \cite{YZ}, here  we need to analyze \eqref{2.1.5} in a finite interval $(-1,1)$ so that the boundary conditions for both $y=1$ and $y=-1$ need to be taken into account. For this,  we construct two independent approximate solutions to the Lees-Lin equation \eqref{2.1.5} at low frequency $k\ll 1$. We start from $k=0$. In this case, the equation \eqref{2.1.5} reduces to
\begin{align}
	\pa_y\left\{ A^{-1}\left[ (U_s-c)\pa_y\varphi-U_s'\varphi \right]\right\}=0,~y\in (-1,1),\nonumber
\end{align}
and it
admits following two independent solutions
\begin{align}
	\varphi_{+}(y)&=U_s(y)-c,\label{2.1.5-2}\\
	\varphi_-(y)&=\left(U_s(y)-c\right)\int_0^y \frac{1}{(U_s(x)-c)^2}\dd x-m^2\left(U_s(y)-c\right)y.
	\label{2.1.5-1}
\end{align}
Based on $\varphi_{\pm}$, we construct two approximate solutions to \eqref{2.1.5} for $k\ll1$.

Define the approximate Green's function:
\begin{equation}
G(x,y)=-(U_s(x)-c)^{-1} 
	\begin{cases}
	\varphi_+ (y) \varphi_- (x) , ~x<y,\\
	\varphi_+ (x) \varphi_- (y) , ~x>y,
	\end{cases}\nonumber
\end{equation}
and two correctors at $k^2$-order:
\begin{equation}
\varphi_{\pm, k}(y)\eqdef \int^1_{-1}G(x,y)(U_s(x)-c)\varphi_\pm (x)\dd x.\label{2.A.0}
\end{equation}
Set 
\begin{equation}
\varphi^s_\pm =\varphi_\pm+k^2\varphi_{\pm, k}.\label{2.A.1}
\end{equation}
Plugging \eqref{2.A.1} into \eqref{2.1.5} gives
\begin{equation}
\text{Ray}_{\text{CNS}}(\varphi^s_\pm)=-k^4(U_s-c)\varphi_{\pm,k}.\label{2.A.2}
\end{equation}
Thus, $\varphi^s_\pm$ solves the Lees-Lin equation (2.1.5) up to $O(k^4)$.

For later use, in the following lemma we summarize the boundary values  of $\varphi^s_\pm$ at $y=\pm1$.
\begin{lemma}\label{lem2.1}
	There exists a positive constant $\gamma_1$, such that for $k\leq\gamma_1$ and $|c|\leq\gamma_1$, the boundary values of $\varphi^s_\pm$ have the following asymptotic behavior: 
	\begin{align}
		\varphi^s_+(-1)&=-c+k^2\frac{1}{U^\prime_s(-1)}\int^1_{-1}U^2_s(x)\dd x+O(1)
k^2|c\log{Im c}|,\label{2.2.2}\\
		\partial_y \varphi^s_+(-1)&=U^\prime_s(-1)+O(1)k^2|\log{Im c}|,\label{2.2.3}\\
		\varphi^s_+(1)&=-c\left(1+O(1)k^2\right),\label{2.2.4}\\
		\partial_y\varphi^s_+(1)&=U^\prime_s(1)+O(1)k^2,\label{2.2.5}\\
		\varphi^s_-(\pm1)&=-\frac{1}{U^\prime_s(\pm1)}+O(1)\bigg(|c\log{Im c}|+k^2\bigg),\label{2.2.6}\\
		\partial_y\varphi^s_-(\pm1)&=O(1)|\log{Im c}|.\label{2.2.8}
	\end{align}
\end{lemma}

\begin{proof}
Consider $\varphi_{+}^s$ first. By the definition of $\varphi_+^s$ in  \eqref{2.A.1} and using \eqref{2.A.0}, we obtain the following explicit formula:
\begin{equation}
	\varphi^s_+(y)=\varphi_+(y)-k^2\varphi_+(y)\bigg(\int^y_{-1}\varphi_+(x)\varphi_-(x)\dd x\bigg)-k^2\varphi_-(y)\bigg(\int^1_y\varphi_+^2(x)\dd x\bigg).\label{2.2.1}
	\end{equation}
Evaluating \eqref{2.2.1} at $y=-1$ and using the boundary value $\varphi_-(-1)$ given in \eqref{ap1}, it holds that
\begin{align}
	\varphi^s_+(-1)&=-c-k^2\varphi_-(-1)\bigg(\int^1_{-1}(U_s(x)-c)^2\dd x\bigg)\nonumber\\
	&=-c-k^2\bigg(\frac{-1}{U^\prime_s(-1)}\int^1_{-1}U^2_s(x)\dd x+O(1)|c\log\text{Im }c|\bigg)\nonumber\\
	&=-c+\frac{k^2}{U^\prime_s(-1)}\int^1_{-1}U^2_s(x)\dd x+O(1)k^2|c\log\text{Im }c|,\nonumber
\end{align}
which is \eqref{2.2.2}.
Similarly, we evaluate $\varphi_+^s(y)$ at $y=1$ to obtain
\begin{align}
	\varphi^s_+(1)&=-c+ck^2\bigg[\int^1_{-1}(U_s(x)-c)\varphi_-(x)\dd x\bigg]\nonumber\\
	&=-c\left(1+O(1)k^2\right).\nonumber
\end{align}
Here in the last line we have used the inequality
$\left|\int_{-1}^1(U_s-c)\varphi_-(x)\dd x\right|\leq C\|\varphi_-\|_{L^\infty}\leq C,
$
with the aid of \eqref{ap13}. Thus, \eqref{2.2.4} follows.

Next we compute boundary values of $\pa_y\varphi_{+}^s$ at $y=\pm 1$. Differentiate \eqref{2.2.1} to obtain
\begin{align}
	\partial_y\varphi^s_+(y)=U'_s(y)-k^2U^\prime_s(y)\int^y_{-1}(U_s(x)-c)\varphi_-(x)\dd x-k^2\varphi_-'(y)\int^1_y(U_s(x)-c)^2\dd x.\label{2.2.1-1}
\end{align}
Evaluating \eqref{2.2.1-1} at $y=-1$ and using \eqref{ap2} in Lemma \ref{lemA1}, we deduce that
\begin{align}
	\partial_y\varphi^s_+(-1)&=U^\prime_s(-1)-k^2\varphi'_-(-1)\bigg(\int^1_{-1}(U_s(x)-c)^2\dd x\bigg)\nonumber\\
	&=U^\prime_s(-1)+O(1)k^2|\log\text{Im} c|,\nonumber
\end{align}
which yields \eqref{2.2.3}. Similarly, we can compute
\begin{align}
	\partial_y\varphi^s_+(1)&=U^\prime_s(1)-k^2U^\prime_s(1)\left(\int^1_{-1}(U_s(x)-c)\varphi_-(x)\dd x\right)\nonumber\\
	&=U^\prime_s(1)-O(1)k^2,\nonumber
\end{align}
which is \eqref{2.2.5}. Therefore, we have completed estimates on boundary data of $\varphi^s_+$.
	
Now we turn to consider $\varphi^s_-$. By definition \eqref{2.A.1}, it holds that
\begin{align}
	\varphi^s_-(y)=\varphi_-(y)-k^2\bigg(\int^y_{-1}\varphi^2_-(x)\dd x\bigg)\varphi_+(y)-k^2\bigg(\int^1_y\varphi_+(x)\varphi_-(x)\dd x\bigg)\varphi_-(y).\label{2.2.10}
\end{align}
Then evaluating $\varphi_{-}^s$ at $y=-1$ and using \eqref{ap1} again, we have
\begin{align}
	\varphi^s_-(-1)&=\varphi_-(-1)\left(1-k^2\int^1_{-1}\varphi_+(x)\varphi_-(x)\dd x\right)\nonumber\\
	&=\bigg(\frac{-1}{U^\prime_s(-1)}+O(1)|c\log\text{Im} c|\bigg)\left(1+O(1)k^2\right)\nonumber\\
	&=\frac{-1}{U^\prime_s(-1)}+O(1)\left(|c\log{\text{Im} c}|+k^2\right).\label{2.2.11-1}
\end{align}
Similarly, the value of $\varphi_-^s(1)$ can be computed as follows.
\begin{align}
	\varphi^s_-(1)&=\varphi_-(1)-k^2\bigg(\int^1_{-1}\varphi^2_-(x)\dd x\bigg)\varphi_+(1)\nonumber\\
	&
	=\frac{-1}{U_s'(1)}+O(1)\left(|c\log\text{Im}c|+k^2\right).\label{2.2.11-2}
\end{align}
Combining \eqref{2.2.11-1} and \eqref{2.2.11-2} together yields  \eqref{2.2.6}.

Finally we estimate the boundary values of derivative $\partial_y\varphi^s_-$.
Differentiating \eqref{2.2.10}, we obtain
\begin{align}\label{2.2.13}
	\partial_y\varphi^s_-(y)=\varphi^\prime_-(y)-k^2\bigg(\int^y_{-1}\varphi^2_-(x)\dd x\bigg)\varphi^\prime_+(y)-k^2\bigg(\int^1_y\varphi_+(x)\varphi_-(x)\dd x\bigg)\varphi^\prime_-(y).
\end{align}
Evaluating \eqref{2.2.13} at $y=-1$ and using Lemma \ref{lemA1}, we have
\begin{align}
	\partial_y\varphi^s_-(-1)&=\varphi^\prime_-(-1)\bigg(1-k^2\int^1_{-1}\varphi_-(x)\varphi_+(x)\dd x\bigg)=O(1)|\log\text{Im} c|.\nonumber
\end{align}
Similarly, we have
\begin{align}
	\partial_y\varphi^s_-(1)&=\varphi^\prime_-(1)-k^2\bigg(\int^1_{-1}\varphi^2_-(x)\dd x\bigg)U^\prime_s(1)
	=O(1)|\log\text{Im}c|.\nonumber
\end{align}
Combining these two equalities gives \eqref{2.2.8}. Therefore, we have shown \eqref{2.2.2}-\eqref{2.2.8}. The proof of Lemma \ref{lem2.1} is completed.
\end{proof}
The following estimates on $\varphi_{\pm}^s$ can be derived 
from  explicit formula in \eqref{2.1.5-2}-\eqref{2.A.1} and bounds on $\varphi_-$ given in Lemma \ref{lemA3}.
\begin{lemma}\label{lem2.3}
Let parameters $k=T_0\eps^{\f17}$ and $c$ satisfies \eqref{c}. The approximate Lees-Lin solutions $\varphi_\pm^s$ satisfy the following uniform estimates with respect to $\epsilon$.
	\begin{align}
	&\|\varphi_+^s\|_{W^{2,\infty}}+\eps^{\f17}\|\pa_y^3\varphi_+^s\|_{L^2}+\eps^{\f27}\|\pa_y^3\varphi_+^s\|_{L^\infty}\leq C,\label{2.2.31}\\
	&	\|\varphi_{-}^s\|_{L^\infty}+ \frac{\|\pa_y\varphi_-^s\|_{L^\infty}}{\left|\log \eps\right|}+\|\pa_y\varphi_-^s\|_{L^2}+\sum_{j=2}^4\eps^{\f{2j-3}{7}}\|\pa_y^j\varphi_-^s\|_{L^2}+\sum_{j=2}^4\eps^{\f{2(j-1)}{7}}\|\pa_y^j\varphi_-^s\|_{L^\infty}\leq  C.\label{2.2.30}
	\end{align}
\end{lemma}
\begin{proof}
	First we estimate $\varphi_+^s$. Differentiating \eqref{2.2.1} yields
	\begin{equation}\label{2.2.31-1}
	\begin{aligned}
		\pa_y\varphi_+^s&=\varphi_+'-k^2\varphi_+'\left(\int_{-1}^y\varphi_+\varphi_-\dd x\right)-k^2\varphi'_-\left(\int_y^1\varphi_+^2\dd x\right),\\
		\pa_y^2\varphi_+^s&=\varphi''_+-k^2\varphi_+''\left(\int_{-1}^y\varphi_+\varphi_-\dd x\right)-k^2\varphi_-''\left(\int_y^1\varphi_+^2\dd x\right)-k^2\varphi'_+\varphi_+\varphi_-+k^2\varphi_-'\varphi_+^2,\\
		\pa_y^3\varphi_+^s&=\varphi_+'''-k^2\varphi_+'''\left(\int_{-1}^y\varphi_+\varphi_-\dd x\right)-k^2\varphi_-'''\left(\int_y^1\varphi_+^2\dd x\right)\\
		&\quad+2k^2\varphi_-''\varphi_+^2-2k^2\varphi_+''\varphi_+\varphi_-+k^2\varphi_+\varphi_+'\varphi_-' -k^2(\varphi_+')^2\varphi_-.
	\end{aligned}
\end{equation}
By the explicit formula \eqref{2.2.31-1}, using bounds $\|\pa_y^j\varphi_+\|_{L^\infty}\leq C$, for $ j=0,1,2,3,$ and \eqref{ap13} in Lemma \ref{lemA3},
we can obtain, for $k\approx \eps^{\f17}$ and $\text{Im}c\approx \eps^{\f27}$, that
\begin{align}
	\|\varphi_{+}^s\|_{W^{2,\infty}}&\leq C\left(1+k^2\|\varphi_-\|_{W^{2,\infty}}\right)\leq C\left(1+k^2|\text{Im}c|^{-1}\right)\leq C,\nonumber\\
		\|\pa_y^3\varphi_{+}^s\|_{L^2}&\leq C\left(1+k^2\|\varphi_-\|_{H^3}\right)\leq C\left(1+k^2|\text{Im}c|^{-\f32}\right)\leq C\eps^{-\f17},\nonumber\nonumber\\
			\|\pa_y^3\varphi_{+}^s\|_{L^\infty}&\leq C\left(1+k^2\|\varphi_-\|_{W^{3,\infty}}\right)\leq C\left(1+k^2|\text{Im}c|^{-2}\right)\leq C\eps^{-\f27}.\nonumber
\end{align}
Combining these inequalities together yields \eqref{2.2.31}.
Similarly, for $\varphi_-^s$, we differentiate \eqref{2.2.10} up to the  fourth order and deduce that
\begin{equation}\label{2.2.31-2}
	\begin{aligned}
		\pa_y\varphi_-^s&=\varphi_-'-k^2\left(\int_{-1}^y\varphi_-^2\dd x\right)\varphi_+'-k^2\left(\int_{y}^1\varphi_+\varphi_-\dd x\right)\varphi_-',\\
		\pa_y^2\varphi_-^s&=\varphi''_--k^2\varphi_+''\left(\int_{-1}^y\varphi_-^2\dd x\right)-k^2\varphi_-''\left(\int_y^1\varphi_+\varphi_-\dd x\right)-k^2\varphi'_+\varphi_-^2+k^2\varphi_-'\varphi_+\varphi_-,\\
		\pa_y^3\varphi_-^s&=\varphi_-'''-k^2\varphi_+'''\left(\int_{-1}^y\varphi_-^2\dd x\right)-k^2\varphi_-'''\left(\int_y^1\varphi_+\varphi_-\dd x\right)\\
		&\quad-2k^2\varphi_+''\varphi_-^2+2k^2\varphi_-''\varphi_+\varphi_-+k^2\varphi_+(\varphi_-')^2 -k^2\varphi_-'\varphi_+'\varphi_-,\\
		\pa_y^4\varphi_-^s&=\pa_y^4\varphi_--k^2\pa_y^4\varphi_+\left(\int_{-1}^y\varphi_-^2\dd x\right)-k^2\pa_y^4\varphi_-\left(\int_y^1\varphi_+\varphi_-\dd x\right)-3k^2\varphi_+'''\varphi_-^2+3k^2\varphi_-'''\varphi_+\varphi_-\\
		&\quad-5k^2\varphi_+''\varphi_-'\varphi_-+4k^2\varphi_-''\varphi_-'\varphi_++k^2\varphi_-''\varphi_+'\varphi_-.
	\end{aligned}
\end{equation}
Then from the explicit formula in \eqref{2.2.31-2}, it holds, for $j=0,1,2,3,4$, that
\begin{align}\label{2.2.31-3}
	\|\pa_y^j\varphi_-^s\|_{L^p}\leq C\left(\|\varphi_-\|_{W^{j,p}}+1\right),~p=2~ \text{or}~ \infty.
\end{align}
The estimate \eqref{2.2.30} follows from \eqref{2.2.31-3} and \eqref{ap13}. Therefore, the proof of Lemma \ref{lem2.3} is completed.
\end{proof}
Based on the approximate solutions $\varphi^s_{\pm}$ to the Lees-Lin equation \eqref{2.1.5}, we can define the following two inviscid modes: $\vec{\Xi}^s_{\pm,\text{app}}=(\rho^s_{\pm,,\text{app}},u^s_{\pm,,\text{app}},v^s_{\pm,,\text{app}})$, where
\begin{align}
	v^s_{\pm,\text{app}}&=-ik\varphi^s_{\pm},\label{2.2.14}\\
	\rho^s_{\pm,\text{app}}&=-m^2A^{-1}(y)\bigg[(U_s-c)\partial_y\varphi^s_{\pm}-\varphi^s_{\pm}U^\prime_s\bigg],\label{2.2.15}\\
	u^s_{\pm,\text{app}}&=\partial_y\varphi^s_{\pm}-(U_s-c)\rho^s_{\pm}.\label{2.2.16}
\end{align}
From \eqref{2.A.2}, we can see that $\vec{\Xi}_{\pm,\text{app}}^s$ solve the inviscid system \eqref{2.1.1} up to $k^4$.

From Lemma \ref{lem2.1}, the following asymptotic expansions of  boundary data for $u_{\pm,\text{app}}^s$ and $v_{\pm,\text{app}}^s$ are obtained.
\begin{lemma}\label{lem2.2}
Under the same assumption on $k$ and $c$ as in Lemma \ref{lem2.1}, it holds that
	\begin{align}
		v_{+,\text{app}}^s(-1)&=ik\left(c-\frac{k^2}{U_s'(-1)}\int_{-1}^1U_s^2(x)\dd x\right)+O(1)k^3|c\log\text{Im}c|,\label{2.2.17}\\
		v_{+,\text{app}}^s(1)&=ikc+O(1)k^3|c|, \label{2.2.18}\\
		v_{-,\text{app}}^s(\pm1)&=\frac{ik}{U_s'(\pm1)}+O(1)k\left(|c|+k^2\right)|\log\text{Im}c|,\label{2.2.19}\\
		u_{+,\text{app}}^{s}(\pm 1)&=U_s'(\pm 1)+O(1)\left(|c|+k^2\right)|\log \text{Im}c|,\label{2.2.21}\\
		u_{-,\text{app}}^s(\pm 1)&=O(1)|\log \text{Im}c|.\label{2.2.23}
	\end{align}
\begin{proof}
	With Lemma \ref{lem2.1} for the boundary values of $\varphi_\pm^s$, the expansions \eqref{2.2.17}-\eqref{2.2.19} for $v_{\pm,\text{app}}^s$ directly follow from \eqref{2.2.14}. Next we show \eqref{2.2.21} and \eqref{2.2.23}.  From Lemma \ref{lem2.1}, the relation \eqref{2.2.15} and boundary conditions $U_s(\pm 1)=0$, we have
	\begin{align}
		\rho_{+,\text{app}}^s(\pm 1)&=m^2A^{-1}(\pm 1)\left[-c\pa_y\varphi_+^s(\pm 1)-\varphi_+^s(\pm 1)U_s'(\pm 1)\right]\nonumber\\
		&=O(1)\left( k^2+|c|\right),\label{E1}
\end{align}
and 
\begin{align}\rho_{-,\text{app}}^s(\pm 1)&=m^2A^{-1}(\pm 1)\left[-c\pa_y\varphi_-^s(\pm 1)-\varphi_-^s(\pm 1)U_s'(\pm 1)\right]=O(1).\label{E2}
	\end{align} 
Thus from \eqref{2.2.16} for $u_{+,\text{app}}^s$ and the boundary values \eqref{2.2.3}, \eqref{2.2.5} and \eqref{E1}, we obtain
\begin{align}
	u_{+,\text{app}}^s(\pm1)&=\pa_y\varphi_+^s(\pm 1)+O(1)|c|\left|\rho_{+,\text{app}}^s(\pm 1)\right|\nonumber\\
	&=U_s'(\pm 1)+\left(k^2+|c|\right)|\log\text{Im}c|,\nonumber
\end{align}
which implies \eqref{2.2.21}. Similarly, from \eqref{2.2.16} for $u_{-,\text{app}}^s$, \eqref{2.2.8} and \eqref{E2}, there holds
\begin{align}
	u_{-,\text{app}}^s(\pm1)&=\pa_y\varphi_-^s(\pm 1)+O(1)|c|\left|\rho_{-,\text{app}}^s(\pm 1)\right|=O(1)|\log \text{Im}c|.\nonumber
\end{align}
This yields \eqref{2.2.23}. Therefore, the proof of Lemma \ref{lem2.2} is completed.
\end{proof}
\end{lemma}
Using Lemma \ref{lem2.3} and \eqref{2.2.14}-\eqref{2.2.16}, we can obtain the following estimates on $\vec{\Xi}_{\pm,\text{app}}^s.$
\begin{corollary}\label{cor2.1} Under the same assumptions on $k$ and $c$ as in Lemma \ref{lem2.3}, the inviscid modes $\vec{\Xi}^s_{\pm,\text{app}}=(\rho^s_{\pm,,\text{app}},u^s_{\pm,,\text{app}},v^s_{\pm,,\text{app}})$ have the following bounds: 
		\begin{align}
			\|v_{+,\text{app}}^s\|_{W^{2,\infty}}+\left\|\left(m^{-2}\rho_{+,\text{app}}^s,u^s_{+,\text{app}}\right)\right\|_{W^{1,\infty}}+\eps^{\f17}\left\|\left(m^{-2}\pa_y^2\rho_{+,\text{app}}^s,\pa_y^2u^s_{+,\text{app}}\right)\right\|_{L^{2}}\leq C,\label{2.2.24}
		\end{align}
and
\begin{align}
	\|v_{-,\text{app}}^s\|_{H^2}&+\eps^{\f27}\|\pa_y^3v_{-,\text{app}}^s\|_{L^2}+\left\|\left(m^{-2}\rho_{-,\text{app}}^s,u^s_{-,\text{app}}\right)\right\|_{L^2}\nonumber\\
	&+\sum_{j=1}^3\eps^{\f{2j-1}7}\left\|\left(m^{-2}\pa_y^j\rho_{-,\text{app}}^s,\pa_y^ju^s_{-,\text{app}}\right)\right\|_{L^{2}}\leq C.\label{2.2.25}
\end{align}
\end{corollary}
\begin{remark} Note that in the above estimates, compared with $\vec{\Xi}_{+,\text{app}}^s,$ we need bound on one more derivative of   $\vec{\Xi}_{-,\text{app}}^s$ in order to get better estimates on boundary values of remainder for $\vec{\Xi}^{s}_{-\text{app}}$. See Section 4 for details.
\end{remark}
\begin{proof} The bounds for $v_{\pm,\text{app}}^s$ are straightforward.
	That is, by  using \eqref{2.2.14} and bounds for $\varphi_\pm^s$ in Lemma \ref{lem2.3}, we  obtain
\begin{align}
&\|v_{+,\text{app}}^s\|_{W^{2,\infty}}\leq C\|\varphi_{+}^s\|_{W^{2,\infty}}\leq C,\label{E3}\\
&\|v_{-,\text{app}}^s\|_{H^2}+\eps^{\f27}\|\pa_y^3v_{-,\text{app}}^s\|_{L^2}\leq \|\varphi_-^s\|_{H^2}+\eps^{\f37}\|\pa_y^3\varphi_-^s\|_{L^2}
\leq C.\label{E4}
\end{align}
For the density and tangential velocity component, we have from the relation \eqref{2.2.15} and \eqref{2.2.16} that $$m^{-2}\rho_{\pm,\text{app}}^s,~u_{\pm,\text{app}}^s \approx |\pa_y\varphi_\pm^s|+|\varphi_\pm^s|.$$
Thus
\begin{align}
	\left\|\left(m^{-2}\rho_{+,\text{app}}^s,u_{+,\text{app}}^s\right)\right\|_{W^{1,\infty}}\leq C\|\varphi_{+}^s\|_{W^{2,\infty}}\leq C,\nonumber
\end{align}
and 
\begin{align}
	\left\|\left(m^{-2}\pa_y^2\rho_{+,\text{app}}^s,\pa_y^2u_{+,\text{app}}^s\right)\right\|_{L^2}\leq C\|\varphi_{+}^s\|_{H^3}\leq C\eps^{-\f17}\nonumber.
\end{align}
Combining  above estimates  with \eqref{E3} yields the bound \eqref{2.2.24} for  $\vec{\Xi}_{+,\text{app}}^s$.
Similarly, for $\rho_{-,\text{app}}^s$ and $u_{-,\text{app}}^s$, we have 
\begin{align}
\left\|\left(m^{-2}\rho_{-,\text{app}}^s,u^s_{-,\text{app}}\right)\right\|_{L^2}\nonumber&+\sum_{j=1}^3\eps^{\f{2j-1}7}\left\|\left(m^{-2}\pa_y^j\rho_{-,\text{app}}^s,\pa_y^ju^s_{-,\text{app}}\right)\right\|_{L^{2}}\\
	&\leq  C\|\varphi_{-}^s\|_{H^1}+C\sum_{j=2}^4\eps^{\f{2j-3}7}\|\pa_y^j\varphi_{-}^s\|_{L^2}\leq C.\nonumber
\end{align}
Here we have used \eqref{2.2.30} in the last inequality. Combining this with \eqref{E4} for $v_{-,\text{app}}^s$ gives \eqref{2.2.25}. Therefore, the proof of Corollary \ref{cor2.1} is completed.
\end{proof}
\subsection{Boundary layers}
To match the exact boundary conditions in \eqref{1.4-1}, we need to construct two boundary layers near $y=\pm 1.$  Firstly, consider the boundary layer on the bottom $y=-1$. For this, we introduce the fast variable
$$z=\frac{1+y}{\delta}$$
where $\delta>0$ is the scale of boundary layer to be determined later. The approximate   solution to \eqref{1.4} is  in the following form
\begin{align}
\Xi_{-,\text{app}}^f=(\ur^{-},\uu^-,ik\delta\uv^-)(z).\label{2.3.-1}
\end{align}
Substituting the ansatz \eqref{2.3.-1} into \eqref{1.4}, we obtain the following system in the new variable $z$:
\begin{equation}\label{2.3.0}
\left\{\begin{aligned}
	&(U_s-c)\ur^-+\uu^-+\partial_z\uv^-=0,\\
	&-\partial^2_z\uu^-+\frac{ik\delta^2}{\epsilon}(U_s-c)\uu^-+\frac{ik\delta^3}{\epsilon}\uv^-U'_s\\
	&\qquad+\frac{\delta^2}{\epsilon}(ikm^{-2}+\epsilon U^{\prime\prime})\ur^-+(\lambda+1)k^2\delta^2\uu^-+\lambda k^2\delta^2\partial_z\uv^-=0,\\
	&-\partial^2_z\uv^-+\frac{ik\delta^2}{\epsilon}(U_s-c)\uv^-+\frac{m^{-2}}{i k \epsilon}\partial_z\ur^-+k^2\delta^2\uv^--\lambda\partial_z(\uu^-+\partial_z\uv^-)=0.
\end{aligned}\right.
\end{equation}
To obtain the leading order profile $(\ur^-,\uu^-,\uv^-)$ near $y=-1$, we expand $U_s(y)$ as follows:
\begin{align}
	U_s(y)-c&=U_s(-1)-c+U_s'(-1)(1+y)+O(1)(1+y)^2\nonumber\\
	&=U^\prime_s(-1)\delta z-c+O(1)|\delta z|^2\nonumber\\
	&=\delta U^\prime_s(-1)(z+z_0)+O(1)|\delta z|^2,\label{2.3.1}
\end{align}
where the boundary condition $U_s(-1)=0$ of Poiseuille flow has been used. Here,  
\begin{align}\label{2.3.3}
	z_0=\frac{-c}{U^\prime_s(-1)\delta}.
\end{align}
We can also expand $U_s'(y)$ at $y=-1$ as
\begin{align}
	U_s'(y)&=U_s'(-1)+O(1)|\delta z|.\label{2.3.2}
\end{align}

Now we plug \eqref{2.3.1} and \eqref{2.3.2} into \eqref{2.3.0} and derive that the scale
\begin{align}
\delta=\frac{e^{-\f{1}{6}\pi i}}{\left[U_s'(-1)n\right]^{\f13}},\text{ where }n=\frac{k}{\eps},\label{2.3.3-1}
\end{align}
 to balance the first three terms in the second equation in \eqref{2.3.0}. Then by taking the leading order of \eqref{2.3.0},  the following system for boundary layer profile $(\ur^-,\uu^-,\uv^-)$ is derived.
\begin{align}
	\ur^-=0,~ \uu^-+\partial_z \uv^-=0, ~-\partial^2_z \uu^-+(z+z_0)\uu^-+\uv^-=0.\label{2.3.4}
	\end{align}
From the above system, we can obtain the following equation for $\uv^-$.
\begin{align}
-\partial^4_z\uv^-+(z+z_0)\partial^2_z\uv^-=0.\label{2.3.5}
\end{align}
As in  \cite{GGN1,GMM1,YZ}, to solve \eqref{2.3.5} we use the function $\text{Ai}(z)$ which is the solution to the Airy equation:
$$\partial^2_z \text{Ai}-z\text{Ai}=0.$$
The first and second order primitive functions of $\text{Ai}(z)$ are denoted by $\text{Ai}(1,z)$ and $\text{Ai}(2,z)$ respectively. They satisfy
$$\partial_z \text{Ai}(2,z)=\text{Ai}(1,z),\text{ and } \partial_z \text{Ai}(1,z)=\text{Ai}(z). $$
Moreover, $\text{Ai}(z)$, $\text{Ai}(1,z)$ and $\text{Ai}(2,z)$ decay at infinity along the straight line $e^{\f{1}{6}\pi i}\mathbb{R}_+$. Set
\begin{align*}
\uv^{-}=\frac{\text{Ai}(2,z+z_0)}{\text{Ai}(2,z_0)},~
\uu^-=-\frac{\text{Ai}(1,z+z_0)}{\text{Ai}(2,z_0)}.
\end{align*}
It is straightforward to check that $(\uu^-, \uv^-)$ satisfies \eqref{2.3.4}. In view of \eqref{2.3.-1}, we define the boundary layer profile at $y=-1$ as follows:
\begin{align}\label{2.3.6}
 \vec{\Xi}^f_{-,\text{app}}=	(\rho^f_-,u^f_-,v^f_-)=\left(0,\uu^-, ik\delta \uv^{-}\right)(z),~ ~	z=e^{\frac16\pi i}(1+y)\left[U_s'(-1)n\right]^{\frac13}.
\end{align}

The boundary layer at $y=1$ can be constructed similarly.  Set the fast variable
$$\tilde{z}=\frac{1-y}{\tilde{\delta}},
$$
where the scale \begin{align}\label{d1}
\tilde{\delta}=\frac{e^{-\f{1}{6}\pi i}}{\left[-U_s'(1)n\right]^{\f13}}
\end{align}
of boundary layer can be derived in the same way as \eqref{2.3.3-1}.
Following \eqref{2.3.-1}-\eqref{2.3.6}, we can construct the  boundary layer profile near $y=1$:
\begin{align}
\vec{\Xi}^f_{+,\text{app}}=(\rho^f_+,u^f_+,v^f_+)=\left(0,\uu^+,ik\tilde{\delta}\uv^+\right)(\tilde{z}),\label{2.3.7}
\end{align}
where
\begin{align}
	\uv^{+}=\frac{\text{Ai}(2,\tilde{z}+\tilde{z}_0)}{\text{Ai}(2,\tilde{z}_0)},~~
	\uu^+=\frac{\text{Ai}(1,\tilde{z}+\tilde{z}_0)}{\text{Ai}(2,\tilde{z}_0)},~~\text{with }\tilde{z}_0=\frac{c}{U_s'(1)\tilde{\delta}}.\label{2.3.7-1}
\end{align}
In the next Lemma, we summarize some pointwise estimates on $\vec{\Xi}_{\pm,\text{app}}^f$. The proof follows from Lemma 3.9 in \cite{LYZ}. Therefore, we omit the details for brevity.

\begin{lemma}\label{lem2.5}
	For $j=0,1,2$ and $l\geq 0$, there exists a positive constant $\tau_1$, such that the boundary layers $\vec{\Xi}_{\pm,\text{app}}^f$ satisfy the following pointwise bounds:
	\begin{align}
		\left|(1\pm y)^l\pa_y^jv_{\pm,\text{app}}^f\right|&\leq C|n|^{\frac{2(j-l)-3}{6}}e^{-\tau_1n^{\f13}(1\pm y)},\label{2.3.8}\\
	\left|(1\pm y)^l\pa_y^ju_{\pm,\text{app}}^f\right|&\leq C|n|^{\frac{j-l}{3}}e^{-\tau_1n^{\f13}(1\pm y)},\label{2.3.9}
	\end{align}
for any $y\in (-1,1).$ Here the constant $C$ is independent of $\eps$.
\end{lemma}
By using explicit formula \eqref{2.3.6}, \eqref{2.3.7} and pointwise estimates \eqref{2.3.8}, \eqref{2.3.9}, we have the following lemma about the asymptotic behavior of boundary values for $(u_{\pm,\text{app}}^f,v_{\pm,\text{app}}^f)$.
\begin{lemma}
	For sufficiently small $\eps$, boundary values of $(u_{\pm,\text{app}}^f,v_{\pm,\text{app}}^f)$ have the following expansions:
	\begin{align}
			&u_{+,\text{app}}^f( 1)=\frac{\text{Ai}(1,\tilde{z}_0)}{\text{Ai}(2,\tilde{z}_0)}, ~ v_{+,\text{app}}^f( 1)=ik\tilde{\delta},~~u_{+,\text{app}}^f(-1),v_{+,\text{app}}^f(-1)=O(1)\eps^{\infty},\label{2.3.10}\\ &u_{-,\text{app}}^f(- 1)=-\frac{\text{Ai}(1,z_0)}{\text{Ai}(2,z_0)}, ~ v_{-,\text{app}}^f(- 1)=ik\delta,~~u_{-,\text{app}}^f( 1), v^f_{-,\text{app}}(1) =O(1)\eps^\infty. \label{2.3.11}
	\end{align}
\end{lemma}
\begin{proof}
	We only show \eqref{2.3.10}. Boundary conditions for $(u_{+,\text{app}}^f,v_{+,\text{app}}^f)$ at $y=1$ directly follow from the explicit formula \eqref{2.3.7} and \eqref{2.3.7-1}. From pointwise estimates \eqref{2.3.8} and \eqref{2.3.9} for $j=l=0,$ we can obtain
	\begin{align}
		\left|v_{+,\text{app}}^f(-1)\right|&\leq Cn^{-\f12}e^{-2\tau_1 n^{\f13}}\leq Ce^{-2\tau_1\eps^{-\f27}},\nonumber\\
	\left|u_{+,\text{app}}^f(-1)\right|&\leq Ce^{-2\tau_1 n^{\f13}}\leq Ce^{-\tau_1\eps^{-\f27}}.\nonumber
	\end{align}
The proof of \eqref{2.3.11} is similar.
\end{proof}
\subsection{Error estimates}

Based on the approximate slow and fast modes constructed in the previous two subsections,  we now summarize the formulation and related estimates on errors generated by the approximations. Recall inviscid modes $\vec{\Xi}^s_{\pm,\text{app}}$ defined in \eqref{2.2.14}-\eqref{2.2.16}, and viscous modes $\vec{\Xi}^f_{\pm,\text{app}}$ given by \eqref{2.3.-1} and \eqref{2.3.7} respectively. Plugging these solutions  into \eqref{1.4}, we  obtain the following error functions: 
\begin{align}\label{2.4.1}
\vec{E}_{\pm}^{s}\eqdef\CL(\vec{\Xi}^s_{\pm,\text{app}})=\bigg(0, &-\eps\Delta_k u_{\pm,\text{app}}^s-\lambda\eps ik\dv(u_{\pm,\text{app}}^s,v_{\pm,\text{app}}^s)+\eps U_s''\rho_{\pm,\text{app}}^s,\nonumber\\
&-\eps\Delta_kv_{\pm,\text{app}}^s-\lambda\eps\pa_y\dv(u_{\pm,\text{app}}^s,v_{\pm,\text{app}}^s)-k^4(U_s-c)\varphi_{\pm,k}\bigg),
\end{align}
where $\varphi_{\pm, k}$ are defined in \eqref{2.A.0},
and
\begin{align}\label{2.4.2}
\vec{E}_{\pm}^f\eqdef	\CL(\vec{\Xi}^f_{\pm,\text{app}})=\bigg(0,~ &k^2\eps u_{\pm,\text{app}}^f+ik\left[U_s(y)-U_s'(\pm 1)(1\pm y)\right]u_{\pm,\text{app}}^f\nonumber\\
	&+\left(U_s'(y)-U_s'(-1)\right)v_{\pm,\text{app}}^f,~
	 -\eps\Delta_kv_{\pm,\text{app}}^f+ik(U_s-c)v_{\pm,\text{app}}^f \bigg).
\end{align}
Based on bounds in Corollary \ref{cor2.1} and Lemma \ref{lem2.5}, we have the following estimates on error terms $\vec{E}^s_{\pm}$ and $\vec{E}^f_{\pm}$. 
\begin{proposition}\label{lem2.4}
Decompose the error function $\vec{E}_+^s$ as
\begin{align}
	\vec{E}_+^s=\vec{E}_{+,1}^s+\vec{E}_{+,2}^s,\label{2.4.4}
\end{align}
where $\vec{E}_{+,2}^s=(0,0,-k^4(U_s-c)\varphi_{+,k}).$  Then it holds that
\begin{align}\label{2.4.4-1} \|\vec{E}_{+,1}^s\|_{L^2}\leq C\eps^{\f67},~ \|\vec{E}_{+,2}^s\|_{H^1}\leq C\eps^{\f47}.
\end{align}
The error function $\vec{E}_-^s$ satifies the following uniform estimate:
\begin{align}
	\|\vec{E}_-^s\|_{L^2}+\eps^{\f27}\|\pa_y\vec{E}_-^s\|_{L^2}\leq C\eps^{\f47}.\label{2.4.3}
\end{align}
For the error functions $\vec{E}_{\pm}^f$ generated by the boundary layers, it holds that
\begin{align}\label{2.4.5}
\|\vec{E}_{\pm}^f\|_{L^2}\leq C\eps^{\f67}.
\end{align}
\end{proposition}
\begin{proof}
	We start with the  error function $\vec{E}_+^s$ defined in \eqref{2.4.1}. By using \eqref{2.2.24} in  Corollary \ref{cor2.1}, we obtain 
	\begin{align}
		\|\vec{E}_{+,1}^s\|_{L^2}\leq C\eps\left\|\left(u_{+,\text{app}}^s,v_{+,\text{app}}^s\right)\right\|_{H^2}+\eps\|\varrho_{+,\text{app}}^s\|_{L^2}\leq C\eps^{\f67},\label{2.4.6}
	\end{align}
and 
\begin{align}
	\|\vec{E}_{+,2}^s\|_{H^1}\leq k^4\left(1+\|\varphi_-\|_{H^1}\right)\leq \eps^{\f47}.\label{2.4.7}
\end{align}
Then combining \eqref{2.4.6} and \eqref{2.4.7} yields  \eqref{2.4.4-1}.

For $\vec{E}_{-}^s$, we use \eqref{2.2.25} to obtain that
\begin{align}
\|\vec{E}_-^s\|_{L^2}\leq C\eps\left\|\left(u_{-,\text{app}}^s,v_{-,\text{app}}^s\right)\right\|_{H^2}+\eps\|\varrho_{-,\text{app}}^s\|_{L^2}+k^4\|\varphi_{-}\|_{L^2}\leq C\eps^{\f47},\label{2.4.8}
\end{align}
and
\begin{align}
	\|\pa_y\vec{E}_{-}^s\|_{L^2}\leq C\eps\left\|\left(u_{-,\text{app}}^s,v_{-,\text{app}}^s\right)\right\|_{H^3}+C\eps\|\varrho_{-,\text{app}}^s\|_{H^1}+Ck^4\|\varphi_-\|_{H^1}\leq C\eps^{\f27}.\label{2.4.9}
\end{align}
Thus the inequality \eqref{2.4.3} follows from \eqref{2.4.8} and \eqref{2.4.9}.

Finally for $\vec{E}_{\pm}^f$, using pointwise estimates \eqref{2.3.8} and \eqref{2.3.9} gives
\begin{align}
	\|\vec{E}_{\pm}^f\|_{L^2}\leq&C k^2\eps\|u_{\pm,\text{app}}^f\|_{L^2}+Ck\|(1\pm y)^2u_{\pm,\text{app}}^f\|_{L^2}+C\|(1\pm y)v_{\pm,\text{app}}^f\|_{L^2}\nonumber\\
	&+C\left(\eps\|v_{\pm,\text{app}}^f\|_{H^2}+k\|(1\pm y)v_{\pm,\text{app}}^f\|_{L^2}+|c|\|v_{\pm,\text{app}}^f\|_{L^2}\right)\nonumber\\
	\leq &C\eps^{\f67},\nonumber
\end{align}
which is \eqref{2.4.5}. The proof of Proposition \ref{lem2.4} is completed.
\end{proof}

\section{Control of remainder}
In the last section, we have constructed four independent{ approximate} solutions to the eigenvalue problem \eqref{1.4}. To correct  errors generated by the approximation, we need to 
solve the following remainder system
\begin{equation}\label{3.1}
	\left\{
	\begin{aligned}
		&ik(U_s-c)\rho+\dv(u,v)=0,\quad y\in(-1,1),\\
		&-\epsilon\Delta_ku-\lambda ik\epsilon \dv(u,v)+ik(U_s-c)u+(ikm^{-2}+\epsilon U^{\prime\prime}_s)\rho+vU^\prime_s=f_u,\\
		&-\epsilon\Delta_k v-\lambda\epsilon\partial_y \dv(u,v)+ik(U_s-c)v+m^{-2}\partial_y\rho=f_v,\\
		&v|_{y=\pm1}=0,
	\end{aligned}
	\right.
\end{equation}
where $(f_u,f_v)$ are in homogenuous source terms.
Notice that in \eqref{3.1}, we only impose the boundary condition for normal velocity $v$. The system \eqref{3.1} is studied in the regime of parameters: $k\approx \eps^{\f17}$ and $c$ satisfies \eqref{c}.

The main result in this section is the following proposition.

\begin{proposition}\label{prop1}
	Let the Mach number $m\in\left(0,\f1{\sqrt{3}}\right)$.
	 If $f_u,f_v\in L^2(-1,1)$, then  remainder system \eqref{3.1} admits a solution $(\rho,u,v)\in H^1(-1,1)\times \left(H^2(-1,1)\right)^2$. Moreover, the solution satisfies  following uniform-in-$\eps$ estimates
	\begin{align}
		\eps^{\f57}\|(m^{-1}\rho,u,v)\|_{L^2}+\eps^{\f57}\|(m^{-2}\pa_y\rho,\pa_yu,\pa_yv)\|_{L^2}+\eps^{\f97}\|(\pa_y^2u,\pa_y^2v)\|_{L^2}\leq C\|(f_u,f_v)\|_{L^2}.\label{3.0.1}
	\end{align}
For $f_u$, $f_v\in H^1(-1,1)$, we define the operator
\begin{align}
	\Omega(f_u,f_v)=-f_v+\frac{1}{ik}\partial_y(A^{-1}f_u).\label{3.0.1-1}
\end{align}
 Then the solution $(\rho,u,v)$ satisfies the following improved estimates:
\begin{align}
	&\eps^{\f27}\|(m^{-1}\rho,u,v)\|_{L^2}+\eps^{\f27}\|(m^{-2}\pa_y\rho,\pa_yu,\pa_yv)\|_{L^2}+\eps^{\f47}\|(\pa_y^2u,\pa_y^2v)\|_{L^2}\nonumber\\
	&\qquad\leq C\|\Omega(f_u,f_v)\|_{L^2}+C\eps^{\f17}\|(f_u,f_v)\|_{L^2}+C\eps^{\f27}\|\dv(f_u,f_v)\|_{L^2}\label{3.0.2}.
\end{align}
Furthermore, the boundary values $u(\pm 1;c)$ are analytic in $c$.
\end{proposition}

\subsection{Quasi-compressible approximation}
To solve \eqref{3.1}, we introduce the following quasi-compressible approximate problem:
\begin{equation}\label{3.1.1}
	\left\{
	\begin{aligned}
		&ik(U_s-c)\varrho+\dv(\cu,\cv)=0,~y\in (-1,1),\\
		&-\epsilon\Delta_k(\cu+(U_s-c)\varrho)+ik(U_s-c)\cu+\cv U^\prime_s+ikm^{-2}\varrho=s_1,\\
		&-\epsilon\Delta_k\cv+ik(U_s-c)\cv+m^{-2}\partial_y\varrho=s_2,\\
		&\cv|_{y=\pm1}=0.
	\end{aligned}
	\right.
\end{equation}
Here $(s_1,s_2)$ is a given source term. We denote the operator on the left hand side of \eqref{3.1.1} by $\fq$.

As in \cite{YZ}, we can decouple $(\varrho,\cu,\cv)$ and reformulate the system \eqref{3.1.1} into an Orr-Sommerfeld type equation. In fact, by the first equation in  \eqref{3.1.1}, it is natural to define the effective stream function $\Phi$, which satisfies:
\begin{align}\label{3.1.2-1}
	\partial_y\Phi=\cu+(U_s-c)\varrho,~ -ik\Phi=\cv,~~ \Phi|_{y=\pm1}=0.
\end{align}
Then plugging \eqref{3.1.2-1}  into the second equation of \eqref{3.1.1}, we can express $\varrho$ in terms of $\Phi$ as 
\begin{align}\label{3.1.2-2}
	m^{-2}\varrho=-A^{-1}(y)\bigg[\frac{i}{n}\Delta_k\partial_y\Phi+(U_s-c)\partial_y\Phi-U_s'\Phi -(ik)^{-1}s_1\bigg],
	\end{align}
where the function $A(y)$ is defined in \eqref{2.1.4}. Substituting \eqref{3.1.2-2} into the third
 equation, we derive the following equation for $\Phi$:
\begin{equation}
	\text{OS}_{\text{CNS}}(\Phi)=\frac{i}{n}\Lambda(\Delta_k\Phi)+(U_s-c)\Lambda(\Phi)-\partial_y(A^{-1}U_s')\Phi=\Omega(s_1,s_2), ~y\in(-1,1).\label{3.1.2}
\end{equation}
Here $n=\frac{k}{\eps}$ is the rescaled frequency, and 
\begin{align*}
	\Lambda: H^2(-1,1)\cap H^1_0(-1,1)\rightarrow L^2(-1,1),\\
	\Lambda(\Phi)=\partial_y(A^{-1}\partial_y\Phi)-k^2\Phi
\end{align*}
is the modified vorticity, and the operator $\Omega$ is defined in \eqref{3.0.1-1}. To solve \eqref{3.1.2}, the following boundary conditions are imposed:
$$\Lambda(\Phi)|_{y=\pm1}=\Phi|_{y=\pm1}=0.$$
If $\Phi$ is a solution to \eqref{3.1.2}, then $(\varrho,\cu,\cv)$ is a solution to the quasi-compressible system \eqref{3.1.1}.

According to the above argument, it is suffcient to study the following Orr-Sommerfeld type equation.
\begin{equation}\label{3.1.3}
	\left\{
	\begin{aligned}
		&\text{OS}_{\text{CNS}}(\Phi)=h,~ y\in(-1,1),\\
		&\Phi |_{y=\pm1}=\Lambda(\Phi)|_{y=\pm1}=0.
	\end{aligned}
	\right.
\end{equation}
Here $h\in L^2(-1,1)$ is a given source. Once the solution $\Phi$ is obtained, the fluid quantities $(\varrho, \cu,\cv)$ can be recovered from \eqref{3.1.2-1} and \eqref{3.1.2-2}.

For the solvability of \eqref{3.1.3}, we have the following lemma.

\begin{lemma}\label{lem3.1}
Let the Mach number $m\in \left(0,\f{1}{\sqrt{3}}\right)$. There exists a unique solution $\Phi\in H^4(-1,1)\cap H^1_0(-1,1)$ to the Orr-Sommerfeld equation \eqref{3.1.3}. Moreover, the solution satisfies 
\begin{align}
	\|\partial_y\Phi, k\Phi\|_{L^2}+\|\Lambda\Phi\|_{L^2}&\leq C\eps^{-\f27}\|h\|_{L^2},\label{A1.1-1}\\
	\|\partial_y\Lambda(\Phi), k\Lambda(\Phi)\|_{L^2}&\leq C\eps^{-\f47}\|h\|_{L^2}.\label{A1.1-2}
\end{align}
\end{lemma}

\begin{proof} We only show the a priori estimates \eqref{A1.1-1} and \eqref{A1.1-2}. The existence and uniqueness parts can be proved in the same way as in Lemma 3.5 of \cite{YZ}.
	Set the following weight function.
\begin{align}\label{A1.1}
	w(y)=-[\partial_y(A^{-1}\partial_yU_s)]^{-1}.
\end{align}
Multiplying \eqref{3.1.3} by $-w\overline{\Lambda(\Phi)}$ and integrating the resultant equation over $(-1,1)$, we have:
\begin{align}\label{A1.2}
\underbrace{-\frac{i}{n}\int^1_{-1}w\Lambda(\Delta_k\Phi)\overline{\Lambda(\Phi)}\dd y}_{I_1}&+\underbrace{\int^1_{-1}-(U_s-c)w|\Lambda(\Phi)|^2\dd y}_{I_2}\nonumber\\
&+\underbrace{\int^1_{-1}-\Phi\overline{\Lambda(\Phi)}\dd y}_{I_3}+\underbrace{\int^1_{-1}hw\overline{\Lambda(\Phi)}\dd y}_{I_4}=0.
\end{align}
Now we estimate $I_i$, $i=1,2,3,4$,  term by term. First we consider $I_1$. 
Denote $[\Delta_k,\Lambda]=\Delta_k(\Lambda(\Phi))-\Lambda(\Delta_k(\Phi))$ and write $I_1$ as
\begin{align}
	I_1&=-\frac{i}{n}\int^1_{-1}w\Delta_k\left(\Lambda(\Phi)\right)\overline{\Lambda(\Phi)}\dd y+\frac{i}{n}\int^1_{-1}w[\Delta_k,\Lambda](\Phi)\overline{\Lambda(\Phi)}\dd y
	=I_{11}+I_{12}.\label{A1.3}
\end{align}
For the first term $I_{11}$, integrating by parts and using boundary conditions $\Lambda(\Phi)|_{y=\pm 1}=0$ give 
\begin{align}\label{A1.3-0}
	I_{11}=\frac{i}{n}\int^1_{-1}w(|\partial_y\Lambda(\Phi)|^2+k^2|\Lambda(\Phi)|^2)\dd y+\frac{i}{n}\int^1_{-1}w^\prime\partial_y\Lambda(\Phi)\overline{\Lambda(\Phi)}\dd y,
\end{align}
in which the last integral is bounded by 
\begin{align}\bigg|\frac{i}{n}\int^1_{-1}w^\prime\partial_y\Lambda(\Phi)\overline{\Lambda(\Phi)}dy\bigg|\leq C\eps^{
\f67}\|\partial_y\Lambda(\Phi)\|_{L^2}\|\Lambda(\Phi)\|_{L^2}.\label{A1.3-4}
\end{align}
Concerning $I_{12}$, we first represent the commutator $[\Delta_k, \Lambda]\Phi$ in terms of $\Lambda(\Phi)$, $\pa_y\Phi$ and $\Phi$:
\begin{align}
	[\Delta_k,\Lambda]\Phi&=\partial^3_y(A^{-1}\partial_y\Phi)-\partial_y(A^{-1}\partial^3_y\Phi)\nonumber\\
	&=2\partial_y(A^{-1})\partial^3_y\Phi+3\partial^2_y(A^{-1})\partial^2_y\Phi+\partial^3_y(A^{-1})\partial_y\Phi.\label{A1.3-1}
\end{align}
Then by using the relations
\begin{align}
\partial^2_y\Phi&=A\Lambda(\Phi)+A^{-1}\partial_y A\partial_y\Phi+k^2A\Phi,\label{A1.3-2}\\
\partial^3_y\Phi&=A\partial_y\Lambda(\Phi)+2\Lambda(\Phi)\partial_yA+\partial_y\Phi(A^{-1}\partial^2_yA+k^2A)+2k^2\Phi\partial_yA,\label{A1.3-3}
\end{align}
we obtain $|\pa_y^2\Phi|\approx |\Lambda(\Phi)|+\left|\left(\pa_y\Phi,k\Phi\right)\right|$, and $|\pa_y^3\Phi|\approx \left|\left(\pa_y\Lambda(\Phi),\Lambda(\Phi)\right)\right|+\left|\left(\pa_y\Phi,k\Phi\right)\right|$. Plugging this  into commutator \eqref{A1.3-1} and taking $L^2$ norm, we obtain
$$\left\|[\Delta_k,\Lambda]\Phi\right\|_{L^2}\leq C\left(\left\|\left(\partial_y\Lambda(\Phi),\Lambda(\Phi)\right)\right\|_{L^2}+\|(\partial_y\Phi,k\Phi)\|_{L^2}\right).$$
Thus, $I_{12}$ is bounded as
\begin{align}
	|I_{12}|&\leq\frac{C}{n}\|[\Delta_k,\Lambda]\Phi\|_{L^2}\|\Lambda(\Phi)\|_{L^2}\nonumber\\
	&\leq C\eps^{\f67}\|\Lambda(\Phi)\|_{L^2}\left(\left\|\left(\partial_y\Lambda(\Phi),\Lambda(\Phi)\right)\right\|_{L^2}+\|(\partial_y\Phi,k\Phi)\|_{L^2}\right).\label{A1.3-5}
\end{align}
By substituting estimates \eqref{A1.3-0}, \eqref{A1.3-4} for $I_{11}$ and \eqref{A1.3-5} for $I_{12}$  into \eqref{A1.3}, then taking real and imaginary parts of the result respectively, we deduce that
\begin{align}
|\text{Re} I_1|\leq& \frac{C\|\text{Im}w\|_{L^\infty}}{n}\left\|\partial_y\Lambda(\Phi)
\right\|^2_{L^2}+C\eps^{\f67}\|\Lambda(\Phi)\|_{L^2}\left(\left\|\left(\partial_y\Lambda(\Phi),\Lambda(\Phi)\right)\right\|_{L^2}+\|(\partial_y\Phi,k\Phi)\|_{L^2}\right)\nonumber\\
\leq &C\eps^{\f87}\left\|\partial_y\Lambda(\Phi)\right\|^2_{L^2}+C\eps^{\f67}\|\Lambda(\Phi)\|_{L^2}\left(\left\|\left(\partial_y\Lambda(\Phi),\Lambda(\Phi)\right)\right\|_{L^2}+\|(\partial_y\Phi,k\Phi)\|_{L^2}\right),\label{A1.4}
\end{align}
and
\begin{align}
\text{Im} I_1\geq&\frac{1}{n}\int_{-1}^1\text{Re}w\left|\left(\pa_y\Lambda(\Phi),k\Lambda(\Phi)\right)\right|^2\dd y-C\eps^{\f67}\|\Lambda(\Phi)\|_{L^2}\bigg(\left\|\left(\partial_y\Lambda(\Phi),\Lambda(\Phi)\right)\right\|_{L^2}+\|(\partial_y\Phi,k\Phi)\|_{L^2}\bigg)\nonumber\\
\gtrsim & \eps^{\f67}\left\|\left(\pa_y\Lambda(\Phi),k\Lambda(\Phi)\right)\right\|^2-C\eps^{\f67}\|\Lambda(\Phi)\|_{L^2}\bigg(\left\|\left(\partial_y\Lambda(\Phi),\Lambda(\Phi)\right)\right\|_{L^2}+\|(\partial_y\Phi,k\Phi)\|_{L^2}\bigg).\label{A1.5}
\end{align}
In \eqref{A1.4} and \eqref{A1.5}, we have used the fact that $|\text{Im}w|\lesssim|c|\lesssim \eps^{\f27}$ and $\text{Re}w\gtrsim 1$ respectively.

Next we estimate $I_2$. By using \eqref{ap7} and \eqref{ap10}, we can obtain
 $$\left|\text{Re}\bigg((U_s-c)w\bigg)\right|\lesssim 1,$$ and
$$\text{Im}\bigg(-(U_s-c)w\bigg)=\left(w_0-U_sw_1\right)\text{Im}c+O(1)|c|^2\gtrsim {\gamma}_0\text{Im} c+O(1)|c|^2\gtrsim \frac{\gamma_0}{2}\eps^{\f27}.$$
Therefore, by taking real and imaginary part of $I_2$ respectively, we deduce that:
\begin{align}
	|\text{Re} I_2|&\leq \left\|\text{Re}\bigg((U_s-c)w\bigg)\right\|_{L^\infty}\|\Lambda(\Phi)\|^2_{L^2}\leq C\|\Lambda(\Phi)\|_{L^2},\label{A1.6}
\end{align}
and
\begin{align}
	\text{Im} I_2&\geq \int_{-1}^1 \text{Im}\bigg(-(U_s-c)w\bigg) \left|\Lambda(\Phi)\right|^2\dd y\geq \frac{\gamma_0}{2}\eps^{\f27}\|\Lambda(\Phi)\|_{L^2}^2.
	\label{A1.7}
\end{align}

For $I_3$, integrating by parts and using the boundary conditions $\Phi|_{y=\pm 1}=0$ give
\begin{align}
	I_3=\int_{-1}^1\b{A}^{-1}|\pa_y\Phi|^2+k^2|\Phi|^2\dd y.\nonumber
\end{align}
Then we have
$$\b{A}^{-1}=(1-m^2U_s^2)^{-2}\left(1-m^2U_s^2-2m^2U_s\b{c}+O(1)|c|^2\right),
$$
which implies that $\text{Re}\left(\b{A}^{-1}\right)
\sim 1$ and $\text{Im}\left(\b{A}^{-1}\right)\gtrsim \left(m^2\text{Im}c\right)U_s-|c|^2$.
Thus, we obtain
\begin{align} \text{Re}I_3&=\int_{-1}^1\text{Re}\left(\b{A}^{-1}\right)|\pa_y\Phi|^2+k^2|\Phi|^2\dd y\gtrsim \|(\pa_y\Phi,k\Phi)\|_{L^2}^2,
	\label{A1.9}
\end{align}
and
\begin{align}
	\text{Im}I_3&=\int_{-1}^1\text{Im}\left(\b{A}^{-1}\right)|\pa_y\Phi|^2\dd y\gtrsim m^2\text{Im}c\left(\int_{-1}^1 U_s|\pa_y\Phi|^2\dd y\right)-|c|^2\|\pa_y\Phi\|_{L^2}^2\nonumber\\
	&\gtrsim -\eps^{\f47}\|\pa_y\Phi\|_{L^2}^2.\label{A1.10}
\end{align}
Finally, by Cauchy-Schwarz inequality, it holds that:
$$|I_4|\lesssim\|h\|_{L^2}\|\Lambda(\Phi)\|_{L^2}.$$
Thus, we have completed the estimates for $I_1$ to $I_4$. By plugging them into \eqref{A1.2} and taking real and imaginary parts respectively, we obtain
\begin{align}\label{A1.11}
	\|(\pa_y\Phi,k\Phi)\|_{L^2}^2\leq& C\|\Lambda(\Phi)\|_{L^2}^2+C\eps^{\f87}\left\|\pa_y\Lambda(\Phi)\right\|_{L^2}^2+C\|h\|_{L^2}^2\nonumber\\
	&+C\eps^{\f67}\|\Lambda(\Phi)\|_{L^2}\left(\left\|\left( \pa_y\Lambda(\Phi),\Lambda(\Phi) \right)\right\|_{L^2}+\|(\pa_y\Phi,k\Phi)\|_{L^2}
	\right)\nonumber\\
	\leq &C\|\Lambda(\Phi)\|_{L^2}^2+C\eps^{\f87}\left\|\pa_y\Lambda(\Phi)\right\|_{L^2}^2+C\|h\|_{L^2}^2,
\end{align}
and
\begin{align}\label{A1.12}
&\eps^{\f47}\left\|\left( \pa_y\Lambda(\Phi),k\Lambda(\Phi) \right)\right\|_{L^2}^2+\|\Lambda(\Phi)\|_{L^2}^2\nonumber\\
&\quad\leq \eps^{\f47}\|\Lambda(\Phi)\|_{L^2}\left(\left\|\left( \pa_y\Lambda(\Phi),\Lambda(\Phi) \right)\right\|_{L^2}+\|(\pa_y\Phi,k\Phi)\|_{L^2}
\right)\nonumber\\
&\qquad+\eps^{\f27}\|\pa_y\Phi\|_{L^2}^2+\eps^{-\f27}\|h\|_{L^2}\|\Lambda(\Phi)\|_{L^2}.\nonumber\\
&\quad\leq o(1)\|(\pa_y\Phi,k\Phi)\|_{L^2}^2+\eps^{-\f47}\|h\|_{L^2}^2.
\end{align}
Set
\begin{align}
	\mathfrak{B}:=\eps^{\f47}\left\|\left( \pa_y\Lambda(\Phi),k\Lambda(\Phi) \right)\right\|_{L^2}^2+\|\Lambda(\Phi)\|_{L^2}^2+\|(\pa_y\Phi,k\Phi)\|_{L^2}^2.\nonumber
\end{align}
 Combining \eqref{A1.11} and \eqref{A1.12} suitably, we have:
 \begin{align}
 	\mathfrak{B}\leq o(1)\mathfrak{B}+\eps^{-\f47}\|h\|_{L^2}.\nonumber
 \end{align}
For sufficiently small $\eps\ll 1$, it holds that $$\mathfrak{B}\leq C\eps^{-\f47}\|h\|_{L^2}^2.$$
This implies inequalities \eqref{A1.1-1} and \eqref{A1.1-2}. The proof of  lemma \ref{lem3.1}   is completed.
\end{proof}
With the solution $\Phi$ to Orr-Sommerfeld type equation \eqref{3.1.2}, we can recover fluid quantities $(\varrho,\cu,\cv)$ in the following corollary.
\begin{corollary}\label{cor1}
For any $s_1,s_2\in H^1(-1,1)$,  there exists a solution $(\varrho,\cu,\cv)\in H^2(-1,1)$ to the quasi-compressible system \eqref{3.1.1}. The solution satisfies the following estimates:
\begin{align}\label{A1.13}
\|\cu\|_{H^1}&+\|(m^{-2}\varrho,\cv)\|_{H^2}+\eps^{-\f17}\|\dv(\cu,\cv)\|_{H^1}\nonumber\\
&\leq C\eps^{-\f27}\|\Omega(s_1,s_2)\|_{L^2}+C\eps^{-\f17}\|s_1\|_{L^2}+C\|s_2\|_{L^2}+C\|\dv(s_1,s_2)\|_{L^2},
\end{align}
and
\begin{align}\label{A1.14}
\|\partial^2_y\cu\|_{L^2}\leq C\eps^{-\f47}\|\Omega(s_1,s_2)\|_{L^2}+C\eps^{-\f17}\|s_1\|_{L^2}+C\|s_2\|_{L^2}+C\|\dv(s_1,s_2)\|_{L^2}.
\end{align}
\end{corollary}

\begin{proof}
Let $\Phi$ be the solution to \eqref{3.1.2}. From \eqref{A1.1-1}, \eqref{A1.1-2}, \eqref{A1.3-2} and \eqref{A1.3-3}, we obtain
\begin{align}
	\|\partial^2_y\Phi\|_{L^2}\leq&C\|\Lambda(\Phi)\|_{L^2}+C\|(\partial_y\Phi,k\Phi)\|_{L^2}\leq C\eps^{-\f27}\|\Omega(s_1,s_2)\|_{L^2},\label{A1.15}
\end{align}
and
\begin{align}
	\|\partial^3_y\Phi\|_{L^2}&\leq C\|\partial_y\Lambda(\Phi)\|_{L^2}+C\|\Lambda(\Phi)\|_{L^2}+C\|(\partial_y\Phi,k\Phi)\|_{L^2}\leq C\eps^{-\f47}\|\Omega(s_1,s_2)\|_{L^2}.\label{A1.16}
\end{align}
From $\cv=-ik\Phi$, it holds that
\begin{align*}
	\|\cv\|_{H^2}&\leq C\|\partial^2_y\Phi\|_{L^2}+C\|(\partial_y\Phi,k\Phi)\|_{L^2}\leq C\eps^{-\f27}\|\Omega(s_1,s_2)\|_{L^2}.
\end{align*}
This gives \eqref{A1.13} for $\cv$. For the density $\varrho$, by using the explicit formula \eqref{3.1.2-2} and the elementary inequality
\begin{align}\label{w1}
	\|\Phi\|_{L^\infty}\leq\int^1_{-1}|\Phi^\prime(x)|\dd x\leq\sqrt{2}\|\Phi^\prime\|_{L^2},\text{ for any }\Phi\in H^1_0(-1,1),
	\end{align}
 we have
\begin{align}
	m^{-2}\|\varrho\|_{L^2}&\leq\frac{C}{n}\|\partial^3_y\Phi\|_{L^2}+C(1+\frac{k^2}{n})\|\partial_y\Phi\|_{L^2}+C\|\Phi\|_{L^\infty}\|\partial_yU_s\|_{L^2}+\frac{C}{k}\|s_1\|_{L^2}\nonumber\\
	&\leq\frac{C}{n}\|\partial^3_y\Phi\|_{L^2}+C\|\partial_y\Phi\|_{L^2}+\frac{C}{k}\|s_1\|_{L^2}\nonumber\\
	&\leq C\eps^{-\f27}\|\Omega(s_1,s_2)\|_{L^2}+C\eps^{-\f17}\|s_1\|_{L^2}.\label{A1.17}
\end{align}
Here we have used \eqref{A1.15} and \eqref{A1.16} in the last inequality of \eqref{A1.17}. For the  derivatives of $\varrho$, we firstly observe that
\begin{align*}
	-m^{-2}\partial_y\varrho&=\text{OS}_{\text{CNS}}(\Phi)+k^2\left[\frac{i}{n}\Delta_k\Phi+(U_s-c)\Phi\right]-\frac{1}{ik}\partial_y(A^{-1}s_1)\\
	&=-s_2+k^2\left[\frac{i}{n}\Delta_k\Phi+(U_s-c)\Phi\right].
\end{align*}
Then taking $L^2$-norm on  both side of the above  equation and using \eqref{A1.15} give
\begin{align}
	m^{-2}\|\partial_y\varrho\|_{L^2}&\leq C\|s_2\|_{L^2}+C\|\partial^2_y\Phi\|_{L^2}+Ck\|\Phi\|_{L^2}\nonumber\\
	&\leq C\eps^{-\f27}\|\Omega(s_1,s_2)\|_{L^2}+C\|s_2\|_{L^2}.\label{A1.18}
\end{align}
Moreover, by taking divergence of the last two equations in \eqref{3.1.1}, we can derive the following equation:
$$-m^{-2}\Delta_k\varrho=-\dv(s_1,s_2)+k^2(U_s-c)^2\varrho+2k^2\Phi U^\prime_s.$$
Thus taking $L^2$-norm and  using \eqref{A1.15} and \eqref{A1.17} yield that
\begin{align}
	m^{-2}\|\partial^2_y\varrho\|_{L^2}&\leq C\|\dv(s_1,s_2)\|_{L^2}+C(1+m^{-2})\|\varrho\|_{L^2}+Ck\|\Phi\|_{L^2}\nonumber\\
	&\leq C\|\dv(s_1,s_2)\|_{L^2}+ C\eps^{-\f27}\|\Omega(s_1,s_2)\|_{L^2}+C\eps^{-\f17}\|s_1\|_{L^2}.\label{A1.19}
\end{align}
Combining \eqref{A1.17}, \eqref{A1.18} and \eqref{A1.19} together yields \eqref{A1.13} for $\varrho$. Concerning about the divergence field $\dv(\cu,\cv)$, we have from the continuity equation $\eqref{3.1.1}_1$ that
$$\eps^{-\f17}\|\dv(\cu,\cv)\|_{H^1}\leq C\|\varrho\|_{H^1}\leq C\eps^{-\f27}\|\Omega(s_1,s_2)\|_{L^2}+C\eps^{-\f17}\|s_1\|_{L^2}+C\|s_2\|_{L^2}.$$

Finally, we estimate $\cu$. From \eqref{3.1.2-1}, it holds that
$|\cu|\sim |\pa_y\Phi|+|\varrho|$. Thus,  by using \eqref{A1.15}, \eqref{A1.16} and \eqref{A1.17}-\eqref{A1.19}, we  obtain
\begin{align}
	\|\cu\|_{H^1}\leq C\|\partial_y\Phi\|_{H^1}+C\|\varrho\|_{H^1}
	\leq C\eps^{-\f27}\|\Omega(s_1,s_2)\|_{L^2}+C\eps^{-\f17}\|s_1\|_{L^2}+C\|s_2\|_{L^2},\nonumber
\end{align}
and
\begin{align*}
	\|\partial^2_y\cu\|_{L^2}&\leq C\|\partial^3_y\Phi\|_{L^2}+C\|\varrho\|_{H^2}\leq C\eps^{-\f47}\|\Omega(s_1,s_2)\|_{L^2}+C\eps^{-\f17}\|s_1\|_{L^2}+C\|s_2\|_{L^2}+C\|\dv(s_1,s_2)\|_{L^2}.
\end{align*}
Therefore,  the estimates \eqref{A1.13} and \eqref{A1.14} for $\cu$ are obtained.  
This completes the proof of the corollary.
\end{proof}
It is straightforward to see that the quasi-compressible operator $\fq$ generates the error
\begin{align}\label{Eq}
	\vec{E}_{\fq}(\varrho,\cu,\cv)&\eqdef \CL(\varrho,\cu,\cv)-\fq(\varrho,\cu,\cv)\nonumber\\
	&=\bigg(0,\eps\Delta_k\left((U_s-c)\varrho\right)-\lambda\eps ik \dv(\cu,\cv)+\eps U_s''\varrho,-\eps\lambda\pa_y\dv(\cu,\cv)\bigg),
\end{align}
which is not in $H^1(-1,1)$. This fact prevents us from iterating $\fq$ directly to solve \eqref{3.1}. In the next subsection, we will introduce the Stokes operator to recover the regularity.
\subsection{Stokes approximation}

To smooth out the error term \eqref{Eq} generated by the quasi-compressible system \eqref{3.1.1} , we introduce the following Stokes approximation:
\begin{equation}\label{A2.1}
	\left\{
	\begin{aligned}
		&ik(U_s-c)\CP+\dv(\CU,\CV)=q_0, ~y\in (-1,1),\\
		&-\epsilon\Delta_k\CU-\lambda ik\epsilon \dv(\CU,\CV)+ik(U_s-c)\CU+(ikm^{-2}+\epsilon U^{\prime\prime}_s)\CP=q_1,\\
		&-\epsilon\Delta_k\CV-\lambda\epsilon\partial_y \dv(\CU,\CV)+ik(U_s-c)\CV+m^{-2}\partial_y\CP=q_2,\\
		&\partial_y\CU|_{y=\pm1}=\CV|_{y=\pm1}=0.
	\end{aligned}
	\right.
\end{equation}
Here $(q_0,q_1,q_2)$ is any given inhomogenuous source term. Compared with the original resolvent problem \eqref{3.1}, we remove the stretching term $\CV U_s'$ in the Stokes system \eqref{A2.1}. The solvability of Stokes system is summarized as follows.

\begin{lemma}\label{prop3}
The Stokes system \eqref{A2.1} admits a unique solution $(\CP,\CU,\CV)\in H^1(-1,1)\times(H^2(-1,1))^2$. Moreover, $(\CP,\CU,\CV)$ satisfies the following estimates:
\begin{align}
	\|(m^{-1}\CP,\CU,\CV)\|_{L^2}&\leq C\eps^{-\f37}\|(m^{-1}q_0,q_1,q_2)\|_{L^2},\label{3.2.0-1}\\
	\|(\partial_y \CU, k\CU)\|_{L^2}+\|(\partial_y\CV,k\CV)\|_{L^2}&\leq C\eps^{-\f57}\|(m^{-1}q_0,q_1,q_2)\|_{L^2},\label{3.2.0-2}\\
	\|div_k(\CU,\CV)\|_{H^1}+m^{-2}\|\partial_y\CP\|_{L^2}&\leq C\eps^{-\f27}\|(m^{-2}q_0,q_1,q_2)\|_{L^2}+C\|\pa_yq_0\|_{L^2},\label{3.2.0-3}\\
	\|(\Delta_k\CU,\Delta_k\CV)\|_{L^2}&\leq C\eps^{-\f97}\|(m^{-1}q_0,q_1,q_2)\|_{L^2}+C\|\pa_yq_0\|_{L^2}.\label{3.2.0-4}
\end{align}
\end{lemma}
\begin{remark}
	In view of \eqref{3.2.0-1}-\eqref{3.2.0-3}, the divergence $\dv(\CU,\CV)$ and density $\CP$ have stronger estimates than other components. This indicates the absence of boundary layers for these two components.
\end{remark}
\begin{proof} We follow the proof in \cite{YZ}. Multiplying the second and third equations by $-\bar{\CU}$ and $-\bar{\CV}$ respectively and integrating by parts, we obtain:
\begin{align}
	&\epsilon\bigg(\|\partial_y\CU,k\CU\|^2_{L^2}+\|\partial_y\CV,k\CV\|^2_{L^2}\bigg)+\lambda\epsilon\|\dv(\CU,\CV)\|^2_{L^2}\nonumber\\
	&+\underbrace{ik\int^1_{-1}(U_s-c)\bigg(|\CU|^2+|\CV|^2\bigg)\dd y}_{J_1}+\underbrace{m^{-2}\int^1_{-1}-\CP\overline{\dv(\CU,\CV)}\dd y}_{J_2}\nonumber\\
	&\qquad=\underbrace{-\int^1_{-1}(q_1+\epsilon U_s''\CP)\bar{\CU}+q_2\bar{\CV}\dd y}_{J_3}.\label{3.2.1}
\end{align}
The real part of $J_1$ gives
\begin{align}\label{A2.4}
	\text{Re}J_1=k\text{Im}c\left\|(\CU,\CV)\right\|_{L^2}^2\gtrsim \eps^{\f37}\|(\CU,\CV)\|_{L^2}^2.
\end{align}
For $J_2$, by using the continuity equation $\eqref{A2.1}_1$, we have
$$\overline{\dv(\CU,\CV)}=ik(U_s-\bar{c})\bar{\CP}+\bar{q_0}.$$
Plugging this  into $J_2$ and taking real part give
\begin{align}
	\text{Re}J_2&=-ikm^{-2}\int^1_{-1}(U_s-\bar{c})|\CP|^2\dd y-m^{-2}\int^1_{-1}\CP\bar{q_0}\dd y\nonumber\\
	&\geq k \text{Im}c\|m^{-1}\CP\|^2_{L^2}-m^{-2}\|\CP\|_{L^2}\|q_0\|_{L^2}\nonumber\\
	&\gtrsim \eps^{\f37}\|m^{-1}\CP\|_{L^2}^2-m^{-2}\|\CP\|_{L^2}\|q_0\|_{L^2}.\label{A2.3}
\end{align}
By Cauchy-Schwarz inequality, $J_3$ is bounded by
\begin{align}\label{A2.2}
	|J_3|\lesssim\|(\CU,\CV)\|_{L^2}\|(q_1,q_2)\|_{L^2}+\epsilon\bigg(\|\CP\|^2_{L^2}+\|\CU\|^2_{L^2}\bigg).
\end{align}
Therefore, by substituting \eqref{A2.4}, \eqref{A2.3}, \eqref{A2.2} into \eqref{3.2.1} and then taking the real part, we deduce that
\begin{align*}
	\epsilon\bigg(\|\partial_y\CU,k\CU\|^2_{L^2}&+\|\partial_y\CV,k\CV\|^2_{L^2}\bigg)+\lambda\epsilon\|\dv(\CU,\CV)\|^2_{L^2}+\eps^{\f37}\|(m^{-1}\CP,\CU,\CV)\|_{L^2}^2\\
	&\leq C\|(m^{-1}\CP,\CU,\CV)\|_{L^2}\|(m^{-1}q_0,q_1,q_2)\|_{L^2}\\
	&\leq o(1)\eps^{\f37}\|(m^{-1}\CP,\CU,\CV)\|_{L^2}^2+C\eps^{-\f37}\|m^{-1}q_0,q_1,q_2)\|_{L^2}^2.
\end{align*}
Absorbing the first term on the right hand side to the left gives \eqref{3.2.0-1} and \eqref{3.2.0-2}.

Next we estimate the density and divergence part of the solution. For this, we denote the vorticity
by 
$$\omega=\partial_y\CU-ik\CV$$
and the divergence of $(\CU,\CV)$  by $\CD=\dv(\CU,\CV)$.  The diffusion in \eqref{A2.1} can be expressed in terms of $\CD$ and $\omega$:
$$\Delta_k\CU=\partial_y\omega+ik\CD,~ \Delta_k\CV=-ik\omega+\partial_y\CD.$$
Thus we can rewrite the momentum equations in \eqref{A2.1} as 
\begin{equation}\label{A2.12}
	\left\{
	\begin{aligned}
		&ikm^{-2}\CP=\epsilon\partial_y\omega+\epsilon(1+\lambda)ik\CD-ik(U_s-c)\CU-q_1-\epsilon U^{\prime\prime}_s\CP,\\
		&m^{-2}\partial_y\CP=-\epsilon ik\omega+\epsilon(1+\lambda)\partial_y\CD-ik(U_s-c)\CV-q_2.
	\end{aligned}
	\right.
\end{equation}
By taking inner product of the first and second equations with $\overline{ik{\CP}}$ and $\partial_y\bar{\CP}$ respectively, we obtain:
\begin{align}
	m^{-2}&\|(\partial_y\CP,k\CP\|^2_{L^2}=\underbrace{\epsilon\int^1_{-1}\partial_y\omega \overline{ik{\CP}}-ik\omega\partial_y\bar{\CP}\dd y}_{J_4}+\underbrace{\epsilon(1+\lambda)\int^1_{-1}k^2\CD\bar{\CP}+\partial_y\CD\partial_y\bar{\CP}\dd y}_{J_5}\nonumber\\
	&+\underbrace{\int^1_{-1}(q_1+\epsilon U^{\prime\prime}_s\CP)ik\bar{\CP}-q_2\partial_y\bar{\CP}\dd y}_{J_6}+\underbrace{\int^1_{-1}ik(U_s-c)(ik\bar{\CP}\CU-\CV\partial_y\bar{\CP})\dd y}_{J_7}.\label{A2.7}
\end{align}
Integrating by parts and using the boundary condition $\omega|_{y=\pm 1}=0$ give
\begin{align}\label{A2.5}
	J_4=-ik\epsilon \b{\CP} \omega\big|_{y=-1}^{y=1}=0.
\end{align}
From the continuity equation $\eqref{A2.1}_1$, it holds that
$\CD=-ik(U_s-c)\CP+q_0,$
and $	\partial_y\CD=-ik(U_s-c)\partial_y\CP-ikU^\prime_s\CP+\partial_yq_0.$ Plugging these  into $J_5$, we obtain
\begin{align}
	J_5=&-\epsilon(1+\lambda)\int^1_{-1}ik(U_s-c)(|\partial_y\CP|^2+k^2|\CP|^2)\dd y\nonumber\\
	&-ik\epsilon(1+\lambda)\int^1_{-1}U^\prime_s\CP\partial_y\CP \dd y+\epsilon(1+\lambda)\int^1_{-1}\left(\partial_y\bar{\CP}\partial_yq_0+k^2q_0\bar{\CP}\right)\dd y.\label{A2.6}
\end{align}
By Cauchy-Schwarz and Young's inequalities, the last two terms on the right hand side of $J_5$ are bounded by 
$$\bigg|ik\epsilon(1+\lambda)\int^1_{-1}U^\prime_s\CP\partial_y\CP dy\bigg|\leq C\epsilon\|(\partial_y\CP,k\CP)\|^2_{L^2},$$
and
\begin{align*}
	\bigg|\epsilon(1+\lambda)\int^1_{-1}\partial_y\bar{\CP}\partial_yq_0+k^2q_0\bar{\CP}dy\bigg|\leq&C\epsilon\|(\partial_y\CP,k\CP)\|_{L^2}\|(\partial_yq_0,kq_0)\|_{L^2}\\
	\leq& \frac{m^{-2}}{8}\|(\partial_y\CP,k\CP)\|^2_{L^2}+Cm^2\epsilon^2\|(\partial_yq_0,kq_0)\|^2_{L^2}.
\end{align*}
Therefore, by taking real part in \eqref{A2.6}, we deduce
\begin{align}
	\text{Re}J_5&\leq-\eps k\text{Im}c(1+\lambda)\|(\partial_y\CP,k\CP)\|^2_{L^2}+C\bigg(\epsilon+\frac{m^{-2}}{8}\bigg)\|(\partial_y\CP,k\CP)\|^2_{L^2}+Cm^2\epsilon^2\|(\partial_yq_0,kq_0)\|^2_{L^2}\nonumber\\
	&\leq\frac{m^{-2}}{4}\|(\partial_y\CP,k\CP)\|^2_{L^2}+Cm^2\epsilon^2\|(\partial_yq_0,kq_0)\|^2_{L^2}.\label{A2.8}
\end{align}
Moreover, by Cauchy-Schwarz and Young's inequalities again, it holds that
\begin{align}
	|J_6|+|J_7|&\leq\bigg(\|q_1,q_2)\|_{L^2}+k\|(\CU,\CV)\|_{L^2}\bigg)\|(\partial_y\CP,k\CP)\|_{L^2}+C\eps^{\f67}\|k\CP\|^2_{L^2}\nonumber\\
	&\leq\frac{m^{-2}}{4}\|(\partial_y\CP,k\CP)\|^2_{L^2}+Cm^2\bigg(\|(q_1,q_2)\|^2_{L^2}+k^2\|(\CU,\CV)\|^2_{L^2}\bigg).\label{A2.9}
\end{align}
By combining the estimates \eqref{A2.5}, \eqref{A2.8} and \eqref{A2.9} for $J_4$ to $J_7$, and taking the real part of \eqref{A2.7}, we obtain that
\begin{align}\label{A2.10}
	m^{-2}\|(\partial_y\CP,k\CP)\|_{L^2}&\leq C\|(q_1,q_2)\|_{L^2}+C\epsilon\|(\partial_yq_0,kq_0)\|_{L^2}+Ck\|(\CU,\CV)\|_{L^2}\nonumber\\
	&\leq C\eps^{-\f27}\|(m^{-1}q_0,q_1,q_2)\|_{L^2}+C\eps\|\pa_yq_0\|_{L^2}.
\end{align}
Here we have used \eqref{3.2.0-1} for $(\CU,\CV)$ in the last line of \eqref{A2.10}.
For the divergence, it holds from $\eqref{A2.1}_1$ and \eqref{A2.10} that
\begin{align}\label{A2.11}
	\|\dv(\CU,\CV)\|_{H^1}&\leq C\|(\partial_y\CP,k\CP)\|_{L^2}+C\|q_0\|_{H^1}\nonumber\\
	&\leq  C\eps^{-\f27}\|(m^{-1}q_0,q_1,q_2)\|_{L^2}+C\|\pa_yq_0\|_{L^2}.
\end{align}
Putting bounds \eqref{A2.10} and \eqref{A2.11} together gives \eqref{3.2.0-3}.

Finally, we estimate second order derivatives of $(\CU,\CV)$. By taking inner product of the first and second equations of \eqref{A2.12} with $\pa_y\b{\omega}$ and $\overline{ik\omega}$ respectively, we can deduce
\begin{align}
	\epsilon\|(\partial_y\omega,k\omega)\|^2_{L^2}=&\int^1_{-1}(q_1+\epsilon U^{\prime\prime}_s\CP)\partial_y\bar{\omega}+q_2\overline{ik\omega}\dd y+\int^1_{-1}ik(U_s-c)(\CU\partial_y\b{\omega}+\CV \overline{ik{\omega}})\dd y\nonumber\\
	\leq&C\bigg(\|(q_1,q_2)\|_{L^2}+\epsilon\|\CP\|_{L^2}+k\|(\CU,\CV)\|_{L^2}\bigg)\|(\partial_y\omega,k\omega)\|_{L^2}.\label{A2.14}
\end{align}
Thus by using \eqref{3.2.0-1} and \eqref{A2.14}, we have
\begin{align}
	\|(\partial_y\omega,k\omega)\|_{L^2}&\leq C\eps^{-1}\|(q_1,q_2)\|_{L^2}+C\eps^{-\f67}\|(m^{-1}\CP,\CU,\CV)\|_{L^2}\nonumber\\
	&\leq C\eps^{-\f97}\|(m^{-1}q_0,q_2,q_2)\|_{L^2}.\label{A2.13}
\end{align}
The estimate \eqref{3.2.0-4} follows from \eqref{A2.11} and \eqref{A2.13}. Therefore, the proof of the lemma  is completed.
\end{proof}

Note that  the error generated by $\fs$ is
\begin{align}\label{Es}
		\vec{E}_\fs(\CP,\CU,\CV)\eqdef\CL(\CP,\CU,\CV)-\fs(\CP,\CU,\CV)
	=\left(0,\CV U_s',0\right).
\end{align}
By Lemma \ref{prop3},  we have $\vec{E}_{\fs}\in H^2(-1,1)$.

\subsection{Solvability of resolvent problem}
In this subsection, we solve the resolvent problem \eqref{3.1} by alternatively iterating the above  Quasi-compressible $\fq$ and Stokes operators $\fs.$  First of all, we construct the solution $(\rho,u,v)$ to \eqref{3.1} for any given inhomogeneous source terms $f_u,f_v\in L^2$, and show the estimate \eqref{3.0.1}.  At the zeroth step, we introduce the Stokes solution $\vec{\Xi}_0=(\CP_0,\CU_0,\CV_0)$, which solves the following system:
 \begin{align}\label{C1.0}
 \fs(\CP_0,\CU_0,\CV_0)=(0,f_u,f_v).
 \end{align}
Recall the error operator $\vec{E}_{\fs}$ defined in \eqref{Es}. The error generated at this step is
\begin{align}\label{C1.1}
	\vec{E}_\fs(\CP_0,\CU_0,\CV_0)
	=\left(0,\CV_0U_s',0\right).
\end{align}
Since $\vec{E}_{\fs}(\CP_0,\CU_0,\CV_0)\in H^2(-1,1)$,   we can introduce the solution $(\varrho_1,\cu_1,\cv_1)$ to the quasi-compressible system:
\begin{align}\label{C1.2}
\fq(\varrho_1,\cu_1,\cv_1)=-\vec{E}_{\fs}(\CP_0,\CU_0,\CV_0).
\end{align}
The error term generated by $\fq$ is 
\begin{align}\label{C1.3}
\vec{E}_{\fq}(\varrho_1,\cu_1,\cv_1)=\left( 0,\eps\Delta_k\left[(U_s-c)\varrho_1\right]-\eps \lambda ik\dv(\cu_1,\cv_1)+\eps U_s''\varrho_1,\eps\lambda\pa_y\dv(\cu_1,\cv_1) \right).
\end{align}
In view of Corollary \ref{cor1}, the right hand side of \eqref{C1.3} is in $L^2(-1,1)$ instead of $H^1(-1,1)$, which forbids us to iterate $\fq$ directly. In order to recover the regularity, we introduce $(\CP_1,\CU_1,\CV_1)$ as the solution to Stokes approximation
\begin{align}\label{C1.4}
\fs(\CP_1,\CU_1,\CV_1)=-\vec{E}_{\fq}.
\end{align}
Now we define the following corrector $\vec{\Xi}_1$ at the Step 1:
\begin{align}\label{C1.5}
\vec{\Xi}_{1}=(\rho_1,u_1,v_1)\eqdef(\varrho_1,\cu_1,\cv_1)+(\CP_1,\CU_1,\CV_1).
\end{align}
One can check that the error term generated at this step is
\begin{align}\label{C1.6}
	\vec{\CE}_1&\eqdef \CL(\vec{\Xi}_0+\vec{\Xi}_1)-(0,f_u,f_v)\nonumber\\
	&=\vec{E}_{\fs}(\CP_1,\CU_1,\CV_1)=(0,\CV_1U_s',0).
\end{align}

Approximate solutions up to any order can be constructed by induction. Suppose that at $N$-th step,  we have the corrector 
\begin{align}\label{C1.7}
	\vec{\Xi}_{N}=(\rho_N,u_N,v_N)\eqdef(\varrho_N,\cu_N,\cv_N)+(\CP_N,\CU_N,\CV_N),
\end{align}
and the error
\begin{align}\label{C1.8}
	\vec{\CE}_N\eqdef \CL\left(\sum_{j=0}^N\vec{\Xi}_j\right)-(0,f_u,f_v)=(0,\CV_NU_s',0).
\end{align}
Then we define the $(N+1)$-th step corrector as 
\begin{align}
\vec{\Xi}_{N+1}=(\rho_{N+1},u_{N+1},v_{N+1})\eqdef(\varrho_{N+1},\cu_{N+1},\cv_{N+1})+(\CP_{N+1},\CU_{N+1},\CV_{N+1}),\label{C1.9}
\end{align}
where $(\varrho_{N+1},\cu_{N+1},\cv_{N+1})$ is the solution to  quasi-compressible system
\begin{align}\label{C1.10}
	\fq(\varrho_{N+1},\cu_{N+1},\cv_{N+1})=-\vec{\CE}_N,
\end{align}
that cancels the error $\vec{\CE}_N$ \eqref{C1.8} generated from $N$-th step, and $(\CP_{N+1},\CU_{N+1},\CV_{N+1})$ solves the Stokes system
\begin{align}\label{C1.11}
	&\fs(\CP_{N+1},\CU_{N+1},\CV_{N+1})=-\vec{E}_{\fq}(\varrho_{N+1},\cu_{N+1},v_{N+1}),
\end{align}
with the error operator $\vec{E}_{\fq}$ being defined in \eqref{Eq}. One can check that the new error generated at $(N+1)$-th step is
\begin{align}\label{C1.12}
	\vec{\CE}_{N+1}&\eqdef \CL\left(\sum_{j=0}^{N+1}\vec{\Xi}_j\right)-(0,f_u,f_v)\nonumber\\
	&=\vec{E}_{\fs}(\CP_{N+1},\CU_{N+1},\CV_{N+1})=(0,\CV_{N+1}U_s',0).
\end{align}
Thus, the approximate solutions up to any order have been constructed. Finally, if  $\vec{\Xi}=\sum_{j=0}^\infty \vec{\Xi}_{j}$ converges, then this series defines a solution to the resolvent problem \eqref{3.1}.

Next, if both $f_u$ and $f_v$ have one order regularity, that is: $(f_u,f_v)\in H^1(-1,1)$, then at the zeroth step we can introduce $(\varrho_0,\cu_0,\cv_0)$ which solves  quasi-compressible system
\begin{align}\label{C1.13}
	\fq(\varrho_0,\cu_0,\cv_0)=(0,f_u,f_v).
\end{align}
The error generated by $(\varrho_0,\cu_0,\cv_0)$ is given by
\begin{align}\label{C1.14}
	\vec{\CE}_{-1}&\eqdef \vec{E}_{\fq}(\varrho_0,\cu_0,\cv_0)\nonumber\\
	&=\left( 0,\eps\Delta_k\left[(U_s-c)\varrho_0\right]-\eps \lambda ik\dv(\cu_{0},\cv_{0})+\eps U_s''\varrho_{0},-\eps\lambda\pa_y\dv(\cu_{0},\cv_{0}) \right).
\end{align}
Since $\vec{\CE}_{-1}\in L^2(-1,1)$, we can repeat the procedure \eqref{C1.1}-\eqref{C1.12} to construct a solution $\vec{\Upsilon}=(\tilde{\rho},\tilde{u},\tilde{v})$ to the resolvent problem $\CL(\vec{\Upsilon})=-\vec{\CE}_{-1}$. Finally, the solution to the original resolvent problem \eqref{3.1} is given by $\vec{\Xi}\eqdef \vec{\Upsilon}+(\varrho_0,\cu_0,\cv_0)$.

We  are  now in the position to prove  the convergence of iteration and solvability of the resolvent problem.

{\bf Proof of Proposition \ref{prop1}}: Recall \eqref{C1.9} the definition of $(N+1)$-step corrector.  First we estimate $\vec{\Xi}_{N+1}=(\varrho_{N+1},\cu_{N+1},\cv_{N+1})$ which solves the  quasi-compressible system \eqref{C1.10} with source term $s_1=\CV_NU_s'$ and $s_2=0$. To use Corollary \ref{cor1}, we compute
\begin{align}
	\Omega(s_1,s_2)=\frac{1}{ik}\pa_y\left(A^{-1}\CV_NU_s'\right)=\frac{1}{ik}\left(\pa_y(A^{-1}U_s')\CV_N+A^{-1}U_s'\pa_y\CV_N\right).\nonumber
\end{align}
By taking $L^2$-norm of $\Omega(s_1,s_2)$, using the elementary inequality \eqref{w1} and the fact $\pa_y\CV_N=\dv(\CU_{N},\CV_{N})-ik\CU_N$, we deduce that
\begin{align}
\|\Omega(s_1,s_2)\|_{L^2}&\leq Ck^{-1}\left(  \|\pa_y(A^{-1}U_s')\|_{L^2}\|\CV_{N}\|_{L^\infty}+\|A^{-1}U_s'\|_{L^\infty}\|\pa_y\CV_{N}\|_{L^2}   \right)\nonumber\\
&\leq Ck^{-1}\|\pa_y\CV_{N}\|_{L^2}\leq C\eps^{-\f17}\|\dv(\CU_N,\CV_N)\|_{L^2}+	C\|\CU_N\|_{L^2}.\label{C2.1}
\end{align}
Similarly, it holds that
\begin{align}\label{C2.2}
	\|s_1\|_{L^2}\leq C\|\CV_N\|_{L^\infty}\leq C\left(\|\dv(\CU_N,\CV_N)\|_{L^2}+\eps^{\f17}\|\CU_N\|_{L^2}\right),
\end{align}
and
\begin{align}\label{C2.3}
	\|\dv(s_1,s_2)\|_{L^2}\leq Ck\|\CV_N\|_{L^\infty}\leq  C\eps^{\f17}\left(\|\dv(\CU_N,\CV_N)\|_{L^2}+\eps^{\f17}\|\CU_N\|_{L^2}\right).
\end{align}
Then by applying Corollary \ref{cor1} to $(\varrho_{N+1},\cu_{N+1},\cv_{N+1})$ and using
the  bounds \eqref{C2.1} to \eqref{C2.3}, we obtain
\begin{align}\label{C2.4}
	&\|\cu_{N+1}\|_{H^1}+\|(m^{-2}\varrho_{N+1},\cv_{N+1})\|_{H^2}+\eps^{-\f17}\|\dv(\cu_{N+1},\cv_{N+1})\|_{H^1}\nonumber\\
	&\quad\leq C\eps^{-\f27} \left(\eps^{-\f17}\|\dv(\CU_N,\CV_N)\|_{L^2}+\|\CU_N\|_{L^2}\right),
\end{align}
and 
\begin{align}\label{C2.5}
	\|\pa_y^2\cu_{N+1}\|_{L^2}\leq C\eps^{-\f{4}{7}}\left(\eps^{-\f17}\|\dv(\CU_N,\CV_N)\|_{L^2}+\|\CU_N\|_{L^2}\right).
\end{align}

Next we estimate $(\CP_{N+1},\CU_{N+1},\CV_{N+1})$  which solves the Stokes system \eqref{C1.11}. Recall the error operator $\vec{E}_{\fq}$ defined in \eqref{Eq}. Then by using the estimate \eqref{C2.4} we obtain
\begin{align}\label{C2.6}
\|\vec{E}_{\fq}(\varrho_{N+1},\cu_{N+1},\cv_{N+1})\|_{L^2}\leq  &C\eps \bigg(\|\varrho_{N+1}\|_{H^2}+\|\dv(\cu_{N+1},\cv_{N+1})\|_{H^1}\bigg)\nonumber\\
\leq&C\eps \bigg(\|m^{-2}\varrho_{N+1}\|_{H^2}+k^{-1}\|\dv(\cu_{N+1},\cv_{N+1})\|_{H^1}\bigg)\nonumber\\
\leq &C \eps^{\f57}\left(\eps^{-\f17}\|\dv(\CU_N,\CV_N)\|_{L^2}+\|\CU_N\|_{L^2}\right).
\end{align}
By applying Lemma \ref{prop3} to $(\CP_{N+1},\CU_{N+1},\CV_{N+1})$ and using the bound \eqref{C2.6}, we deduce 
\begin{align}
		\|(m^{-1}\CP_{N+1},\CU_{N+1},\CV_{N+1})\|_{L^2}&\leq C \eps^{-\f37}\|\vec{E}_{\fq}(\varrho_{N+1},\cu_{N+1},\cv_{N+1})\|_{L^2}\nonumber\\
		&\leq C\eps^{\f27}\left(\eps^{-\f17}\|\dv(\CU_N,\CV_N)\|_{L^2}+\|\CU_N\|_{L^2}\right),\label{C2.7}\\
	\eps^{-\f17}\left(\|\dv(\CU_{N+1},\CV_{N+1})\|_{H^1}+m^{-2}\|\partial_y\CP_{N+1}\|_{L^2}\right)&\leq C\eps^{-\f37}\|\vec{E}_{\fq}(\varrho_{N+1},\cu_{N+1},\cv_{N+1})\|_{L^2}\nonumber\\
		&\leq C\eps^{\f27}\left(\eps^{-\f17}\|\dv(\CU_N,\CV_N)\|_{L^2}+\|\CU_N\|_{L^2}\right),\label{C2.8}\\
	\|\partial_y \CU_{N+1}, k\CU_{N+1}\|_{L^2}+\|\partial_y\CV_{N+1},k\CV_{N+1}\|_{L^2}&\leq C\eps^{-\f57}\|\vec{E}_{\fq}(\varrho_{N+1},\cu_{N+1},\cv_{N+1})\|_{L^2}\nonumber\\
	&\leq C\left(\eps^{-\f17}\|\dv(\CU_N,\CV_N)\|_{L^2}+\|\CU_N\|_{L^2}\right),\label{C2.9}\\
	\|(\Delta_k\CU_{N+1},\Delta_k\CV_{N+1})\|_{L^2}&\leq C\eps^{-\f97}\|\vec{E}_{\fq}(\varrho_{N+1},\cu_{N+1},\cv_{N+1})\|_{L^2}\nonumber\\
	&\leq C\eps^{-\f47}\left(\eps^{-\f17}\|\dv(\CU_N,\CV_N)\|_{L^2}+\|\CU_N\|_{L^2}\right).\label{C2.10}
\end{align}

Now for $N=0,1,2,\cdots$, we define
\begin{align}\label{C2.11}
	E_{N}\eqdef~ \|(m^{-1}\CP_{N},\CU_{N},\CV_{N})\|_{L^2}+\eps^{-\f17}\left(\|\dv(\CU_{N},\CV_{N})\|_{H^1}+m^{-2}\|\partial_y\CP_{N}\|_{L^2}\right).
\end{align}
From \eqref{C2.7} and \eqref{C2.8},  it holds that $E_{N+1}\leq C\eps^{\f27}E_{N}$. Then by  taking $\eps\ll 1$ suffciently small, we have 
\begin{align}\label{C2.12}
\sum_{j=0}^\infty E_j\leq \sum_{j=0}^\infty \left( \f12\right)^j E_0\leq CE_0.
\end{align}
For other components, 
by using the estimates \eqref{C2.4}, \eqref{C2.5}, \eqref{C2.9} and \eqref{C2.10}, we obtain that
\begin{align}
	&\sum_{j=1}^\infty\|\cu_{j}\|_{H^1}+\|(m^{-2}\varrho_{j},\cv_{j})\|_{H^2}+\eps^{-\f17}\|\dv(\cu_{j},\cv_{j})\|_{H^1}\leq C\eps^{-\f27}\left(\sum_{j=0}^\infty E_j\right)\leq C\eps^{-\f27} E_0,\label{C2.13}\\
	&\sum_{j=1}^\infty\|\pa_y^2\cu_{j}\|_{L^2}\leq C\eps^{-\f47}\left(\sum_{j=0}^\infty E_j\right)\leq C\eps^{-\f47}E_0,\label{C2.14}\\
	&\sum_{j=1}^\infty \|(\pa_y\CU_{j},\pa_y\CV_{j})\|_{L^2}\leq C\left(\sum_{j=0}^\infty E_j\right)\leq CE_0,\label{C2.15}\\
	&\sum_{j=1}^\infty \|(\pa_y^2\CU_{j},\pa_y^2\CV_{j})\|_{L^2}\leq C\eps^{-\f47}\left(\sum_{j=0}^\infty E_j\right)\leq C\eps^{-\f47}E_0.\label{C2.16}
\end{align}
Next we estimate
 $\vec{\Xi}_0=(\CP_0,\CU_0,\CV_0)$. Since $\vec{\Xi}_0$ solves the Stokes system \eqref{C1.0}, we can apply Lemma \ref{prop3} to $\vec{\Xi}_0$ with $q_0=0$, $q_1=f_u$, and $q_2=f_v$. Thus, it holds that
\begin{align}
	E_0&\leq C\eps^{-\f37}\|(f_u,f_v)\|_{L^2},\label{C2.17}\\
	\|\partial_y \CU_0, k\CU_0\|_{L^2}+\|\partial_y\CV_0,k\CV_0\|_{L^2}&\leq C\eps^{-\f57}\|(f_u,f_v)\|_{L^2},\label{C2.18}\\
	\|(\pa_y^2\CU_0,\pa_y^2\CV_0)\|_{L^2}&\leq C\eps^{-\f97}\|(f_u,f_v)\|_{L^2}.\label{C2.19}
\end{align}
Combining estimates \eqref{C2.12}-\eqref{C2.19} together gives
\begin{align}
	\|(m^{-1}\rho,u,v)\|_{L^2}&\leq \sum_{j=0}^\infty \|(m^{-1}\CP_j,\CU_j,\CV_j)\|_{L^2}+\sum_{j=1}^\infty\|(m^{-1}\varrho_j,\cu_j,\cv_j)\|_{L^2}\nonumber\\
	&\leq C\left(1+\eps^{-\f27}
	\right)E_0\leq C\eps^{-\f57}\|(f_u,f_v)\|_{L^2},\label{C2.20}\\
	\|m^{-2}\pa_y\rho\|_{L^2}+\|\dv(u,v)\|_{L^2}&\leq \eps^{\f17}\left(\sum_{j=0}^\infty E_j\right)+\sum_{j=1}^\infty\left(\|m^{-2}\pa_y\varrho_j\|_{L^2}+\|\dv(\cu_j,\cv_j)\|_{L^2}\right)\nonumber\\
	&\leq C(1+\eps^{-\f27})E_0\leq C\eps^{-\f57}\|(f_u,f_v)\|_{L^2},\label{C2.21}\\
	\|(\pa_yu,\pa_yv)\|_{L^2}&\leq \sum_{j=0}^\infty\|(\pa_y\CU_j,\pa_y\CV_j)\|_{L^2}+\sum_{j=1}^\infty\|(\pa_y\cu_j,\pa_y\cv_j)\|_{L^2}\nonumber\\
	&\leq C\left(1+\eps^{-\f27}\right)E_0+C\eps^{-\f57}\|(f_u,f_v)\|_{L^2}\leq  C\eps^{-\f57}\|(f_u,f_v)\|_{L^2},\label{C2.22}\\
	\|(\pa_y^2u,\pa_y^2v)\|_{L^2}&\leq \sum_{j=0}^\infty\|(\pa_y^2\CU_j,\pa_y^2\CV_j)\|_{L^2}+\sum_{j=1}^\infty\|(\pa_y^2\cu_j,\pa_y^2\cv_j)\|_{L^2}\nonumber\\
	&\leq C\left(1+\eps^{-\f47}\right)E_0+C\eps^{-\f97}\|(f_u,f_v)\|_{L^2}\leq C\eps^{-\f97}\|(f_u,f_v)\|_{L^2}.\label{C2.23}
\end{align}
By putting \eqref{C2.20}-\eqref{C2.23} together, the inequality \eqref{3.0.1} is obtained.

Finally, we prove the improved estimate \eqref{3.0.2}. Suppose that $f_u$ and $f_v\in H^1(-1,1)$. As mentioned before, we look for a solution in the form  of 
$$
(\rho,u,v)=(\varrho_0,\cu_0,\cv_0)+(\tilde{\rho},\tilde{u},\tilde{v}),
$$
where $(\varrho_0,\cu_0,\cv_0)$ solves the quasi-compressible system \eqref{C1.13} that generates an error  $\vec{\CE}_{-1}$ defined in \eqref{C1.14}, and $(\tilde{\rho},\tilde{u},\tilde{v})$ solves the resolvent problem $\CL(\tilde{\rho},\tilde{u},\tilde{v})=-\vec{\CE}_{-1}$. Now we estimate $(\varrho_0,\cu_0,\cv_0)$. Use Corollary \ref{cor1} to  have
\begin{align}\label{C2.24}
	\|\cu_0\|_{H^1}&+\|(m^{-2}\varrho_0,\cv_0)\|_{H^2}+\eps^{-\f17}\|\dv(\cu_0,\cv_0)\|_{H^1}\nonumber\\
	&\leq C\eps^{-\f27}\|\Omega(f_u,f_v)\|_{L^2}+C\eps^{-\f17}\|(f_u,f_v)\|_{L^2}+C\|\dv(f_u,f_v)\|_{L^2},
\end{align}
and
\begin{align}\label{C2.25}
	\|\partial^2_y\cu_0\|_{L^2}\leq C\eps^{-\f47}\|\Omega(f_u,f_v)\|_{L^2}+\eps^{-\f17}\|(f_u,f_v)\|_{L^2}+C\|\dv(f_u,f_v)\|_{L^2}.
\end{align}
Moreover, to bound $(\tilde{\rho},\tilde{u},\tilde{v})$, we first observe that
\begin{align}
	\|\vec{\CE}_{-1}\|_{L^2}&\leq \|\vec{E}_{\fq}(\varrho_0,\cu_0,\cv_0)\|_{L^2}\leq C\eps\|\varrho_0\|_{H^2}+C\eps\|\dv(\cu_0,\cv_0)\|_{H^1}\nonumber\\
	&\leq C\eps\left(\eps^{-\f27}\|\Omega(f_u,f_v)\|_{L^2}+C\eps^{-\f17}\|(f_u,f_v)\|_{L^2}+C\|\dv(f_u,f_v)\|_{L^2}\right),\label{C2.25-1}
\end{align}
where we have used \eqref{C2.24} in the last inequality. Then by applying \eqref{3.0.1} to $(\tilde{\rho},\tilde{u},\tilde{v})$,
and using the bound \eqref{C2.25-1} on source $\vec{\CE}_{-1}$, we obtain that
\begin{align}\label{C2.26}
	&\|(m^{-1}\tilde{\rho},\tilde{u},\tilde{v})\|_{L^2}+\|(m^{-2}\pa_y\tilde{\rho},\pa_y\tilde{u},\pa_y\tilde{v})\|_{L^2}+\eps^{\f47}\|(\pa_y^2\tilde{\cu},\pa_y^2\tilde{\cv})\|_{L^2}\nonumber\\
	&\qquad\leq C\eps^{-\f57}\|\vec{\CE}_{-1}\|_{L^2}\leq  C\|\Omega(f_u,f_v)\|_{L^2}+C\eps^{\f17}\|(f_u,f_v)\|_{L^2}+C\eps^{\f27}\|\dv(f_u,f_v)\|_{L^2}.
\end{align}
Putting \eqref{C2.24}, \eqref{C2.25} and \eqref{C2.26} together gives
\begin{align}
	&\|(m^{-1}\rho,u,v)\|_{L^2}+\|(m^{-2}\pa_y\rho,\pa_yu,\pa_yv)\|_{L^2}\nonumber\\
	&\qquad\leq C\eps^{-\f27}\|\Omega(f_u,f_v)\|_{L^2}+C\eps^{-\f17}\|(f_u,f_v)\|_{L^2}+C\|\dv(f_u,f_v)\|_{L^2},\label{C2.28}
\end{align}
and
\begin{align}
	\|(\pa_y^2u,\pa_y^2v)\|_{L^2}\leq C\eps^{-\f47}\|\Omega(f_u,f_v)\|_{L^2}+C\eps^{-\f37}\|(f_u,f_v)\|_{L^2}+C\eps^{-\f27}\|\dv(f_u,f_v)\|_{L^2}.\label{C2.29}
\end{align}
Combining \eqref{C2.28} and \eqref{C2.29} together implies \eqref{3.0.2}. The analyticity can be proved  as in \cite{YZ}. We omit the details  for brevity. Therefore, the proof of  the proposition is completed.
\qed

\section{Dispersion relation}
Recall in Section 2.1 and 2.2 that we have four independent approximate solutions:
$$\vec{\Xi}^s_{\pm,\text{app}}\text{ and } \vec{\Xi}^f_{\pm,\text{app}}.$$
Here  $\vec{\Xi}^s_{\pm,\text{app}}$ are inviscid modes which are determined in terms of
solutions  $\varphi^s_\pm$ to the Lees-Lin equation \eqref{2.1.5}, and $\vec{\Xi}^f_{\pm,\text{app}}$ are viscous modes which are boundary layers at $y=\pm 1$.  Using Proposition \ref{prop1}, we can construct four exact independent solutions to \eqref{1.4} near $\vec{\Xi}^s_{\pm,\text{app}}$  and  $\vec{\Xi}^f_{\pm,\text{app}}.$ 
\begin{proposition}\label{prop4.1}
The eigenvalue problem \eqref{1.4} admits four solutions, that is,  $\vec{\Xi}^s_{\pm}=(\rho_\pm^s,u_\pm^s,v_\pm^s)$ and $\vec{\Xi}^f_{\pm}=(\rho_\pm^f,u_\pm^f,v_\pm^f)$. Moreover, their boundary values satisfy the following asymptotic properties:
	\begin{align}
	v_{+}^s(-1)&=ik\left(c-\tau{k^2}\right)+O(1)\eps^{\f57}|\log\eps|, \text{ with } \tau=\f1{U_s'(-1)}\int_{-1}^1U_s^2(x)\dd x,\label{4.0}\\
	v_{+}^s(1)&=ikc+O(1)\eps^{\f57}, \label{4.1}\\
	v_{-}^s(\pm1)&=\frac{ik}{U_s'(\pm1)}+O(1)\eps^{\f37}|\log\eps|,\label{4.2}\\
	u_{+}^{s}(\pm 1)&=U_s'(\pm 1)+O(1)\eps^{\f17},\label{4.3}\\
	u_{-}^s(\pm 1)&=O(1)|\log \eps|,\label{4.4}\\
	u_{+}^f( 1)&=\frac{\text{Ai}(1,\tilde{z}_0)}{\text{Ai}(2,\tilde{z}_0)}+O(1)\eps^{\f17},~ v_+^f( 1)=ik\tilde{\delta},~ v_\pm ^f(-1)=O(1)\eps^{\infty},~ u_\pm^f(-1 )=O(1)\eps^{\f17} \label{4.5},\\
	u_{-}^f(- 1)&=-\frac{\text{Ai}(1,z_0)}{\text{Ai}(2,z_0)}+O(1)\eps^{\f17},~ v_-^f(- 1)=ik\delta,~ v_- ^f(1)=O(1)\eps^{\infty},~ u_-^f(1 )=O(1)\eps^{\f17} \label{4.6}.
\end{align}
Here $z_0$ and $\tilde{z}_0$ are defined in \eqref{2.3.3} and \eqref{2.3.7-1} respectively.
\end{proposition}
\begin{proof}
	First we construct $\vec{\Xi}_{+}^s$. We look for the solution in form of
	\begin{align}
	\vec{\Xi}_+^s=\vec{\Xi}_{+,\text{app}}^s+\vec{\Xi}_{+,r}^s,\nonumber
	\end{align}
where the approximate solution $\vec{\Xi}_{+,\text{app}}^s$ is defined in \eqref{2.2.14}-\eqref{2.2.16}. Recall the error function $\vec{E}_+^s$ in \eqref{2.4.1} and decomposition $\vec{E}_+^s=\vec{E}_{+,1}^s+\vec{E}_{+,2}^s$ in \eqref{2.4.4}. Then we  decompose the remainder accordingly into: $\vec{\Xi}_{+,r}^s=\vec{\Xi}_{+,r,1}^s+\vec{\Xi}_{+,r,2}^s$. Here,  $\vec{\Xi}_{+,r,j}^s=(\rho_{+,r,j}^s,u_{+,r,j}^s, v_{+,r,j}^s)$ satisfies 
\begin{align}
	\CL(\vec{\Xi}_{+,r,j}^s)=-\vec{E}_{+,j}^s,~~v_{+,r,j}^s\big|_{y=\pm 1}=0,~j=1,2.\label{4.1.1}
\end{align}
The solvability of $\vec{\Xi}_{+,r,j}^{s}$ is guaranteed by Proposition \ref{prop1}. By using \eqref{2.4.4-1} and \eqref{3.0.1}, we  obtain
\begin{align}
	\left|u_{+,r,1}^s(\pm 1)\right|\leq \|u_{+,r,1}^s\|_{H^1}\leq C\eps^{-\f57}\|\vec{E}_{+,1}^s\|_{L^2}\leq C\eps^{\f17},\label{4.1.2}
\end{align}
and
\begin{align}
	\left|u_{+,r,2}^s(\pm 1)\right|\leq \|u_{+,r,2}^s\|_{H^1}&\leq C\eps^{-\f27}\left(\|\Omega(\vec{E}_{+,2}^s)\|_{L^2}+\eps^{\f17}\|\vec{E}_{+,2}^s\|_{L^2}+\eps^{\f27}\|\dv(\vec{E}_{+,2}^s)\|_{L^2}\right)\nonumber\\
	&\leq C\eps^{-\f27}\|\vec{E}_{+,2}^s\|_{H^1}\leq C\eps^{\f27}.\label{4.2.3}
\end{align}
From \eqref{2.2.17}, \eqref{2.2.18}, \eqref{2.2.21}, \eqref{4.1.1}, \eqref{4.1.2} and \eqref{4.2.3}, it holds that
\begin{align}
	v_+^s(-1 )&=v_{+,\text{app}}^s(-1)=ik(c-\tau k^2)+O(1)\eps^{\f57}|\log \eps|,\nonumber\\
	v_+^s(1 )&=v_{+,\text{app}}^s(1)=ikc+O(1)\eps^{\f57},\nonumber\\
	u_+^s(\pm 1)&=u_{+,\text{app}}^s(\pm 1)+u_{+,r,1}^s(\pm 1)+u_{+,r,2}^s(\pm 1)=U_s'(\pm 1)+O(1)\eps^{\f17}.\nonumber
\end{align}
Thus we have shown the asymptotic properties given in  \eqref{4.0}, \eqref{4.1} and \eqref{4.3} for the  boundary data of $\vec{\Xi}_{+}^s$.

For $\vec{\Xi}_-^s$, we recall the approximate solution $\vec{\Xi}_{-,\text{app}}^s$ defined in \eqref{2.2.14}-\eqref{2.2.16}. This approximation generates an error  $\vec{E}_{-}^s$ which is given in \eqref{2.4.1}. To get rid of the error, it is natural to seek the solution $\vec{\Xi}_-^s$ in the following form
\begin{align}
	\vec{\Xi}_-^s=\vec{\Xi}_{-,\text{app}}^s+\vec{\Xi}_{-,r}^s,\nonumber
\end{align}
where the remainder
$\vec{\Xi}_{-,r}^s=(\rho_{-,r}^s,u_{-,r}^s, v_{-,r}^s)$ satisfies 
\begin{align}
	\CL(\vec{\Xi}_{-,r}^s)=-\vec{E}_{-}^s,~~v_{-}^s\big|_{y=\pm 1}=0.\label{4.1.3}
\end{align}
By using \eqref{3.0.2} and \eqref{2.4.3}, we obtain that
\begin{align}
	\left|u_{-,r}^s(\pm 1)\right|\leq C \|u_{-,r}^s\|_{H^1}&\leq C\eps^{-\f27}\left(\|\Omega(\vec{E}_{-}^s)\|_{L^2}+\eps^{\f17}\|\vec{E}_{-}^s\|_{L^2}+\eps^{\f27}\|\dv(\vec{E}_{-}^s)\|_{L^2}\right)\nonumber\\
	&\leq C\eps^{-\f27}\|\vec{E}_-^s\|_{H^1}\leq C.\label{4.1.4}
\end{align}
Then by \eqref{2.2.19}, \eqref{4.1.3} and \eqref{4.1.4}, we have
\begin{align}
	v_{-}^s(\pm 1)=v_{-,\text{app}}^s(\pm 1)=\frac{ik}{U_s'(\pm 1)}+O(1)\eps^{\f37}|\log \eps|,\nonumber
\end{align}
and 
\begin{align}
	u_-^s(\pm 1)=u_{-,\text{app}}^s(\pm 1)+u_{-,r}^s(\pm 1)=O(1)|\log \eps|.\nonumber
\end{align}
This completes the proof of \eqref{4.2} and \eqref{4.4}.

Finally we construct the viscous modes $\vec{\Xi}_{\pm}^f$ near the boundary layer profiles $\vec{\Xi}_{\pm,\text{app}}^f$ defined in \eqref{2.3.6} and \eqref{2.3.7} respectively. Without loss of generality, we consider $\vec{\Xi}_+^f$. Recall the error function $\vec{E}_{+}^f$ given in  \eqref{2.4.2}. We look for the solution in form of 
\begin{align}
	\vec{\Xi}_{+}^f=\vec{\Xi}_{+,\text{app}}^f+\vec{\Xi}_{+,r}^f,\nonumber
\end{align}
where $\vec{\Xi}_{+,r}^f$  solves  the following resolvent problem:
\begin{align}
	\CL(\vec{\Xi}_{+,r}^f)=-\vec{E}_{+}^f,~~v_{+,r}^f\big|_{y=\pm 1}=0.\label{4.1.5}
\end{align}
Using \eqref{2.4.5} and \eqref{3.0.1}, we can obtain that
\begin{align}
	|u_{+,r}^f(\pm 1)|\leq C\|u_{+,r}^f\|_{H^1}\leq C\eps^{-\f57}\|\vec{E}_{+}^f\|_{L^2}\leq C\eps^{\f17}.\label{4.1.6}
\end{align}
Then from \eqref{2.3.10}, \eqref{4.1.5} and \eqref{4.1.6},  it holds that
\begin{align}
	&v_{+}^f( 1)=v_{+,\text{app}}^f(1)=ik\tilde{\delta},~v_{+}^f( -1)=v_{+,\text{app}}^f(-1)=O(1)\eps^\infty\nonumber,\\
	&u_{+}^f( 1)=u_{+,\text{app}}^f( 1)+u_{+,r}^f( 1)= \frac{\text{Ai}(1,\tilde{z}_0)}{\text{Ai}(2,\tilde{z}_0)}+O(1)\eps^{\f17},\nonumber\\
	&u_{+}^f(- 1)=u_{+,\text{app}}^f(- 1)+u_{+,r}^f(- 1)=O(1)\eps^{\f17}.\nonumber
\end{align}
Combining these estimates together yields \eqref{4.5}. The asymptotic expansion \eqref{4.6} for $\vec{\Xi}_-^f$ can be proved similarly. Therefore, the proof of  the proposition is completed.
\end{proof}

To match exact boundary conditions at $y=\pm1$, we construct a solution $\vec{\Xi}$ to the eigenvalue problem \eqref{1.4} by the following linear combination of $\vec{\Xi}^s_{\pm}$  and  $\vec{\Xi}^f_{\pm}:$
\begin{align}\label{4.7}
\vec{\Xi}=(\rho,u,v)\eqdef\vec{\Xi}^s_++\alpha_1(c)\vec{\Xi}^s_-+\alpha_2(c)\vec{\Xi}^f_-+\alpha_3(c)\vec{\Xi}^f_+,
\end{align}
where $\alpha_1,\alpha_2,\alpha_3$ are constants. Now we choose $\alpha_1$, $\alpha_2$ and $\alpha_3$ so that the following three boundary conditions are satisfied:
\begin{align}
	v(-1)&=v^s_+(-1)+\alpha_1v^s_-(-1)+\alpha_2v^f_-(-1)+\alpha_3v^f_+(-1)=0,\label{4.8}\\
	u(-1)&=u^s_+(-1)+\alpha_1u^s_-(-1)+\alpha_2u^f_-(-1)+\alpha_3u^f_+(-1)=0,\label{4.9}\\
	v(1)&=v^s_+(1)+\alpha_1v^s_-(1)+\alpha_2v^f_-(1)+\alpha_3v^f_+(1)=0.\label{4.10}
\end{align}
This reduces to solve the following algebraic system:
$$\mathcal{M}(\alpha_1,\alpha_2,\alpha_3)^T=-\bigg(v^s_+(-1),u^s_+(-1),v^s_+(1)\bigg).$$
where the coefficient matrix $\mathcal{M}$ is given by:
\begin{equation*}
	\mathcal{M}=\left(
	\begin{aligned}
		v^s_-(-1)\quad&&v^f_-(-1)\quad&&v^f_+(-1)\\
		u^s_-(-1)\quad&&u^f_-(-1)\quad&&u^f_+(-1)\\
		v^s_-(1)\quad&&v^f_-(1)\quad&&v^f_+(1)
	\end{aligned}
	\right).
\end{equation*}
The following lemma gives the invertibility of $\mathcal{M}$ and asymptotic formula of coefficients $\a_1$, $\a_2$ and $\a_3$ for $\eps\ll 1$. Denote $\beta=\frac{U_s'(-1)}{\left|U_s'(1)\right|}$.
\begin{lemma}\label{lem4.1}
	For any $\eps\ll1$, the matrix $\mathcal{M}$ is invertible. Moreover, the coefficients $\a_1$, $\a_2$ and $\a_3$ in \eqref{4.8}-\eqref{4.10} have the following asymptotic expansions: \begin{align}
		\alpha_1&=-\delta\left[U_s'(-1)\right]^2\frac{\text{Ai}(2,z_0)}{\text{Ai}(1,z_0)}-U_s'(-1)(c-\tau k^2)+O(1)\eps^{\f37},\label{4.17}\\
		\a_2&=U_s'(-1)\frac{\text{Ai}(2,z_0)}{\text{Ai}(1,z_0)}+O(1)\eps^{\f17},\label{4.18}\\
		\alpha_3&=-\delta^{-1}\left[(1+\beta)\beta^{-\f13} c-\tau \beta^{\f23} k^2\right]-U_s'(-1)\beta^{\f23}\frac{\text{Ai}(2,z_0)}{\text{Ai}(1,z_0)}+O(1)\eps^{\f17}.\label{4.19}
	\end{align}
	
\end{lemma}

\begin{proof}
	By using asymptotic expansions of boundary values in Proposition \ref{prop4.1}, we have
	\begin{equation*}
		\mathcal{M}=\left(
		\begin{aligned}
			\frac{ik}{U_s'(-1)}+O(1)\eps^{\f37}|\log\eps|\quad&&ik\delta\quad&&O(1)\eps^{\infty}\\
			O(1)|\log\eps|\quad&&-\frac{\text{Ai}(1,z_0)}{\text{Ai}(2,z_0)}+O(1)\eps^{\f17}\quad&&O(1)\eps^{\f17}\\
			\frac{ik}{U_s'(1)}+O(1)\eps^{\f37}|\log\eps|\quad&&O(1)\eps^{\infty}\quad&&ik\tilde{\delta}
		\end{aligned}
		\right).
	\end{equation*}
Then we  compute the determinant of $\CM$ as follows.
\begin{align}
	\det \mathcal{M}
	=&ik\tilde{\delta}\cdot \det
	\left(
	\begin{aligned}
		\frac{ik}{U_s'(-1)}+O(1)\eps^{\f37}|\log\eps|\quad&&ik\delta\\
		O(1)|\log\eps|\quad&&-\frac{\text{Ai}(1,z_0)}{\text{Ai}(2,z_0)}+O(1)\eps^{\f17}
	\end{aligned}
	\right)\nonumber\\
	&+O(1)\eps^{\f17}\cdot \det
	\left(
	\begin{aligned}
		\frac{ik}{U_s'(-1)}+O(1)\eps^{\f37}|\log\eps|\quad&&ik\delta\\
		\frac{ik}{U_s'(1)}+O(1)\eps^{\f37}|\log\eps|\quad&&O(1)\eps^{\infty}
	\end{aligned}
	\right)+O(1)\eps^{\infty}\nonumber\\
	=&\frac{k^2\tilde{\delta}}{U_s'(-1)}\left(\frac{\text{Ai}(1,z_0)}{\text{Ai}(2,z_0)}-O(1)\eps^{\f17}|\log\eps|\right)\sim k^2\tilde{\delta}\neq0.\nonumber
\end{align}
Therefore, the matrix $\mathcal{M}$ is invertible and we can solve
	$$(\alpha_1,\alpha_2,\alpha_3)=-\mathcal{M}^{-1}\bigg(v^s_+(-1),u^s_+(-1),v^s_+(1)\bigg).$$
To study the asymptotic properties of $\alpha_1,\alpha_2$ and $\alpha_3$, we substitute \eqref{4.0}-\eqref{4.6} into \eqref{4.8}-\eqref{4.10} and  obtain that
\begin{align}
	&\left(\f{ik}{U_s'(-1)}+O(1)\eps^{\f37}|\log\eps|\right)\a_1+ik\delta\a_2+O(1)\eps^{\infty}\a_3=-\left[ik\left(c-\tau k^2\right)+O(1)\eps^{\f57}|\log\eps|\right],\label{4.11}\\
	&O(1)|\log\eps|\a_1+\left(-\frac{\text{Ai}(1,z_0)}{\text{Ai}(2,z_0)}+O(1)\eps^{\f17}\right)\a_2+O(1)\eps^{\f17}\a_3=-U_s'(-1)+O(1)\eps^{\f17},\label{4.12}\\
    &\left(\f{ik}{U_s'(1)}+O(1)\eps^{\f37}|\log\eps|\right)\a_1+O(1)\eps^{\infty}\a_2+ik\tilde{\delta}\a_3=-\left(ikc+O(1)\eps^{\f57}\right).\label{4.13}
\end{align}
Recall that $k\sim \eps^{\f17}$, $\delta\sim |c|\sim \eps^{\f27}$. From \eqref{4.11}-\eqref{4.13}, we can see that $\a_1\sim \eps^{\f27}$, $\a_2\sim O(1)$ and $\alpha_3\sim O(1)$. Then 
by taking the leading order of both sides of equations in \eqref{4.11}-\eqref{4.13}, and using $\tilde{\delta}=\beta^{\f13}\delta$,
 we obtain
\begin{align}
	\alpha_1&=-\delta[U_s'(-1)]^2\frac{\text{Ai}(2,z_0)}{\text{Ai}(1,z_0)}-U_s'(-1)(c-\tau k^2)+O(1)\eps^{\f37},\label{4.15}\\
	\a_2&=U_s'(-1)\frac{\text{Ai}(2,z_0)}{\text{Ai}(1,z_0)}+O(1)\eps^{\f17},\label{4.14}\\
	\alpha_3&=-\delta^{-1}\left[(1+\beta)\beta^{-\f13}c-\tau\beta^{\f23} k^2\right]-U_s'(-1)\beta^{\f23}\frac{\text{Ai}(2,z_0)}{\text{Ai}(1,z_0)}+O(1)\eps^{\f17}.\label{4.16}
\end{align}
This completes the proof of the lemma.
\end{proof}

To match the last boundary condition, that is, $u(y;c)|_{y=1}=0$, we introduce the function:
\begin{align}
	\text{I}(c)\eqdef u(1;c)=u^s_+(1;c)+\alpha_1(c)u^s_-(1;c)+\alpha_2(c)u^f_-(1;c)+\alpha_3(c)u^f_+(1;c),\nonumber
\end{align}
where constants $\alpha_1, \alpha_2$ and $\alpha_3$ are determined in Lemma \ref{lem4.1}. 
It is obvious that if $c$ is a zero point of $\text{I}(c)$, then \begin{align}\vec{\Xi}(y;c)=\vec{\Xi}^s_+(y;c)+\alpha_1\vec{\Xi}^s_-(y;c)+\alpha_2\vec{\Xi}^f_-(y;c)+\alpha_3\vec{\Xi}^f_+(y;c)\label{4.23}
\end{align}
 defines a solution to the eigenvalue problem \eqref{1.4} with desired boundary conditions \eqref{1.4-1}.\\

We are now ready to prove the main result in this paper.\\

{\bf Proof of Theorem \ref{thm1.1}:}  Let 
 $k=T_0\epsilon^\frac17$ where the constant $T_0>1$ is a sufficiently large. We look for the zero point of $\text{I}(c)$ in the disk  
 \begin{align}\label{cd}
 	D=\bigg\{c\in\mathbb{C}~\bigg||c-c_0|\leq T^{-2}_0\epsilon^\frac27\bigg\}
\end{align}
centered at \begin{align}
 	c_0= \frac{\tau\beta T_0^2}{1+\beta}\eps^{\f27}+T_0^{-\f32}e^{\f14\pi i}\frac{\left|U_s'(1)\right|^{\f12}(1+\beta^{\f32})}{\tau^{\f12}\beta^{\f12}(1+\beta)^{\f12}}\eps^{\f27}.\label{c0}
 \end{align}
From Proposition \ref{prop4.1} and \eqref{4.17}-\eqref{4.19} in Lemma \ref{lem4.1}, we 
have 
\begin{align}
	\text{I}(c)=&
	U_s'(1)-U_s'(-1)\beta^{\f23}\frac{\text{Ai}(2,z_0)}{\text{Ai}(1,z_0)}\frac{\text{Ai}(1,\tilde{z}_0)}{\text{Ai}(2,\tilde{z}_0)}\nonumber\\
	&-\delta^{-1}\left[(1+\beta)\beta^{-\f13}c-\tau\beta^{\f23}k^2\right]\frac{\text{Ai}(1,\tilde{z}_0)}{\text{Ai}(2,\tilde{z}_0)}
	+O(1)\eps^{\f17},\label{4.20}
\end{align}
where$z_0$ and $\tilde{z}_0$ are defined in \eqref{2.3.3} and \eqref{2.3.7-1} respectively. Next we compute the leading order terms of $\frac{\text{Ai}(1,\tilde{z}_0)}{\text{Ai}(2,\tilde{z}_0)}$ and $\frac{\text{Ai}(2,{z}_0)}{\text{Ai}(1,{z}_0)}$ as follows.  By taking $T_0>1$ sufficiently large, we can find a positive constant $b_0$ independent of $\eps$ and $T_0$, such that for any $c\in D$, 
it holds
\begin{align}
\text{Im}c\geq b_0T_0^{-\f32}\eps^{\f27},~\arg c\in \left( 0,b_0T_0^{-\f72}\right),~|c|=\frac{\tau\beta T_0^2}{1+\beta}\eps^{\f27}\left(1+O(1)T_0^{-\f72}\right)\label{4.21}.
\end{align}
From \eqref{d1}, we have
\begin{align}
	\tilde{z}_0=-\tilde{\delta}^{-1}c=-e^{\f16\pi i}\left|U_s'(1)\right|^{\f13}T_0^{\f13}\eps^{-\f27}c\nonumber.
\end{align}
If $c$ is in the disk $D$, then from \eqref{4.21} we can obtain 
\begin{align}\nonumber
	|\tilde{z}_0|=\frac{\tau\beta\left|U_s'(1)\right|^{\f13}}{1+\beta} T_0^{\f73}\left(1+O(1)T_0^{-\f{7}2}\right),~ -\f56 \pi<\arg z_0<-\f56 \pi+b_0T_0^{-\f72}.
\end{align}
Then from the asymptotic behavior of the Airy profile (cf. \cite{GGN1,GMM1}), it holds that
\begin{align}
	\frac{\text{Ai}(1,\tilde{z}_0)}{\text{Ai}(2,\tilde{z}_0)}=-\tilde{z}_0^{\f12}+O(1)|\tilde{z}_0|^{-1}=-e^{-\f{5}{12}\pi i}\frac{\tau^{\f12}\beta^{\f12}\left|U_s'(1)\right|^{\f16}}{(1+\beta)^{\f12}}T_0^{\f76}+O(1)T_0^{-\f73}.\label{4.22}
\end{align}
Similarly, we have
\begin{align}\label{4.22-1}
	\frac{\text{Ai}(2,{z}_0)}{\text{Ai}(1,{z}_0)}=\frac{-e^{\frac{5}{12}\pi i}(1+\beta)^{\f12}}{\tau^{\f12}\beta^{\f12}\left[U_s'(-1)\right]^{\f16}}T_0^{-\f76}\left(1+O(1)T_0^{-\f72}\right).
\end{align}
Plugging the asymptotic expansion \eqref{4.22} and \eqref{4.22-1} into \eqref{4.20}, we deduce that
\begin{align}
	\text{I}(c)=\text{I}_{\text{lin}}(c)+O(1)\left(T_0^{-\f{7}{2}}+\eps^{\f17}\right),\nonumber
\end{align}
where the linear function
\begin{align}
	\text{I}_{\text{lin}}(c)=U_s'(1)-U_s'(-1)\beta^{\f12}+e^{-\f14\pi i}\frac{\tau^{\f12}\left|U_s'(1)\right|^{\f12}\beta^{\f56}}{(1+\beta)^{\f12}}T_0^{\f32}\left[(1+\beta)\beta^{-\f13}c-\tau\beta^{\f23}k^2\right].\nonumber
\end{align}
By using the same argument as in \cite{YZ}, one can show that the function $\text{I}(c)$ is analytic in the disk $D$. Moreover, one can check that $c_0$ given in \eqref{c0} is the unique zero point of linear function $\text{I}_{\text{lin}}(c)$. On the circle $\pa D$, we have
\begin{align}
	\left|\text{I}_{\text{lin}}(c)\right|=\tau^{\f12}\left|U_s'(1)\right|^{\f12}\beta^{\f12}(1+\beta)^{\f12}T_0^{-\f12},\nonumber
\end{align}
and
\begin{align}
	\left|\text{I}(c)-	\text{I}_{\text{lin}}(c)\right|\leq C\left(T_0^{-\f72}+\eps^{\f17}\right)\leq	
	 \f12\tau^{\f12}\left|U_s'(1)\right|^{\f12}\beta^{\f12}(1+\beta)^{\f12}T_0^{-\f12}\leq \f12\left|\text{I}_{\text{lin}}(c)\right|.\nonumber
\end{align}
By Rouch\'e's Theorem, the function $\text{I}(c)$ and $\text{I}_{\text{lin}}(c)$ have the same number of zero point in $D$. Thus, there exists a unique $c\in D$,  such that \eqref{4.23} defines a solution of eigenvalue problem \eqref{1.4} with boundary conditions \eqref{1.4-1}. Therefore, the proof of Theorem \ref{thm1.1} is completed.

\section{Appendices}
First of all, we give in the following lemma the boundary values of $\varphi_-$ defined in \eqref{2.1.5-1}.
\begin{lemma}\label{lemA1}
For sufficiently small $|c|$, the boundary values of $\varphi_-(y)$ satisfy:
\begin{align}
\varphi_-(\pm1)&=\frac{-1}{U^\prime_s(\pm1)}+O(1)|c\log{Im c}|,\label{ap1}\\
\varphi^{\prime}_-(\pm1)&=O(1)|\log{Im c}|.\label{ap2}
\end{align}
\begin{proof}
We only show \eqref{ap1} and \eqref{ap2} for $y=1$, since the boundary value at $y=-1$ can be obtained in the same way. Evaluating \eqref{2.1.5-1} at $y=1$, we have
\begin{align}
\varphi_-(1)&=(U_s(1)-c)\int^1_0\frac{1}{(U_s(x)-c)^2}\dd x-m^2(U_s(1)-c)\nonumber\\
&
=-c\int^1_0\frac{1}{(U_s(x)-c)^2}\dd x+m^2c\nonumber\\
&=-c\int^1_{\frac{1}{2}}\frac{1}{(U_s(x)-c)^2}\dd x-c\int^{\frac{1}{2}}_0\frac{1}{(U_s(x)-c)^2}\dd x+O(1)|c|.\label{ap3}
\end{align}
For $x\in \left(0,\frac{1}{2}\right)$, we have $U_s(x)=1-x^2\geq \frac34$.
Then the second integral in \eqref{ap3} satisfies
\begin{equation}
\bigg|\int^{\frac{1}{2}}_0\frac{c}{(U_s(x)-c)^2}\dd x\bigg|\leq\frac{|c|}{2\left(\frac{9}{16}-|c|^2\right)}\\
\leq O(1)|c|, \text{ for } |c|\ll1.\label{ap5}
\end{equation}
When $x\in\left(\frac12,1\right)$, $U$ is not denegerate because that $|U^\prime_s(x)|=2x\ge1$. Thus, we can compute the first integral in \eqref{ap3} as
\begin{align}
&-\int^1_\frac12 \frac{c}{(U_s(x)-c)^2}\dd x=\int^1_\frac12 \frac{\dd}{\dd x}\bigg(\frac{c}{U_s(x)-c}\bigg)\frac{1}{U^\prime_s(x)}\dd x\nonumber\\
&\qquad=\frac{c}{U^\prime_s(1)(U_s(1)-c)}-\frac{c}{U^\prime_s\left(\frac12\right)\left(U_s\left(\frac12\right)-c\right)}-\int^1_\frac12\frac{c}{U_s(x)-c}\bigg(\frac{1}{U^\prime_s(x)}\bigg)^\prime \dd x\nonumber\\
&\qquad=-\frac{1}{U^\prime_s(1)}+c\int^1_\frac12 \frac{\dd}{\dd x}(\log(U_s(x)-c))\frac{U^{\prime\prime}_s(x)}{(U^{\prime}_s(x))^3}+O(1)|c|\nonumber\\
&\qquad=-\frac{1}{U^\prime_s(1)}+c\log(U_s(y)-c)\frac{U^{\prime\prime}_s(y)}{(U^\prime_s(y))^3}\bigg|_{y=\f12}^{y=1}-c\int^1_\frac12\log(U_s(x)-c)\left(\frac{U^{\prime\prime}_s}{(U^\prime_s)^3}\right)^\prime \dd x+O(1)|c|\nonumber\\
&\qquad=-\frac{1}{U^\prime_s(1)}+O(1)|c\log \text{Im} c|.\label{ap6}
\end{align}
Plugging \eqref{ap5} and \eqref{ap6} into \eqref{ap3}, we obtain:
\begin{equation}
\varphi_-(1)=-\frac{1}{U^\prime_s(1)}+O(1)|c \log \text{Im} c|,\nonumber
\end{equation}
which is \eqref{ap1}. Next we show \eqref{ap2}. Differentiating \eqref{2.1.5-1} yields 
\begin{align}\label{ap4}
	\varphi_-'(y)=\frac{U_s'(y)\varphi_-(y)+A(y)}{U_s(y)-c}.
\end{align}
Evaluaing both sides of \eqref{ap4} at $y=1$ gives
\begin{align}
\varphi^\prime_-(1)&=\frac{U^\prime_s(1)\varphi_-(1)+1}{U_s(1)-c}-m^2(U_s(1)-c)\nonumber\\
&=\frac{O(1)|c\log \text{Im} c |}{-c}+m^2c\nonumber\\
&=O(1)|\log \text{Im} c|,\nonumber
\end{align}
which is \eqref{ap2}. The proof of Lemma \ref{lemA1} is completed.
\end{proof}
\end{lemma}

The following lemma concerns some bounds on $\varphi_-.$ 
\begin{lemma}\label{lemA3}
There exists a positive constant $\gamma,$ such that if $c$ lies in the half disk $\{\text{Im}c>0, |c|\leq \gamma\}$, then $\varphi_-$ satisfies the following estimates.
	\begin{itemize}
		\item[(1)] If $y\in \left(-\f12,\f12\right)$, we have
		\begin{align}
			\left|\pa_y^k\varphi_-(y)\right|\leq C,~k=0,1,2,3,4.\label{ap11}
		\end{align}
		\item[(2)] If $y\in (-1,1)\setminus \left(-\f12,\f12\right),$ we have the following pointwise estimate
		\begin{equation}\label{ap12}
			\begin{aligned}
				&\left|\varphi_-(y)\right|\leq C,~\varphi_-'(y)=\frac{-U_s''(y)}{\left(U_s'(y)\right)^2}\log\left(U_s(y)-c\right)+O(1),\\
				&\varphi_-''(y)=-\frac{U_s''(y)}{U_s'(y)\left(U_s(y)-c\right)}-\frac{\left(U_s''(y)\right)^2}{\left(U_s'(y)\right)^3}\log\left(U_s(y)-c\right)+O(1),\\
				&\varphi_-'''(y)=\frac{U_s''(y)}{\left(U_s(y)-c\right)^2}+O(1),~\pa_y^4\varphi_-=-\frac{3U_s'U_s''}{(U_s-c)^3}.
			\end{aligned}
		\end{equation}
		\item[(3)] Furthermore, it holds that
		\begin{equation}\label{ap13}
			\begin{aligned}
				&\|\varphi_-\|_{L^\infty}\leq C,~	\|\varphi_-'\|_{L^\infty}\leq C\left|\log\text{Im}c\right|,~\|\varphi_-'\|_{L^2}\leq C,\\
				&\|\pa_y^j\varphi_-\|_{L^\infty}\leq C\left|\text{Im}c\right|^{-j+1},~j=2,3,4,\\ &\|\pa_y^j\varphi_-\|_{L^2}\leq C\left|\text{Im}c\right|^{-j+\f32},~j=1,2,3,4.
			\end{aligned}
		\end{equation}
	\end{itemize}
\end{lemma}
\begin{proof}
	By differentiating \eqref{2.1.5-1} up to the fourth order, we obtain
\begin{equation}\label{ap14}
	\begin{aligned}
		\varphi_-'&=U_s'\int_0^y\frac{1}{(U_s-c)^2}\dd x+\frac{1}{U_s-c}-m^2\left(yU_s'+U_s-c\right),\\
		\varphi_-''&=U_s''\int_0^y\frac{1}{(U_s-c)^2}\dd x-m^2\left(yU_s''+2U_s'\right),\\
		\varphi_-'''&=\frac{U_s''}{(U_s-c)^2}-3m^2U_s'',~
		\pa_y^4\varphi_-=-\frac{2U_s'U_s''}{(U_s-c)^3}.
	\end{aligned}
\end{equation}
Here we have used $\pa_y^jU_s=0$ when $j=3,4$. For $y\in \left(-\f12,\f12\right)$, $|U_s-c|\gtrsim 1$. Thus \eqref{ap11} follows from the explicit formula \eqref{ap14}. For $ y\in (-1,1)\setminus \left(-\f12,\f12\right)$, as \eqref{ap6} we can obtain the following pointwise estimate
\begin{align}\label{ap15}
	\int_0^y\frac{1}{(U_s-c)^2}\dd x=-\frac{1}{U_s'(y)(U_s(y)-c)}-\log(U_s(y)-c)\frac{U_s''(y)}{(U_s'(y))^3}+O(1).
\end{align} 
Substituting \eqref{ap15} into the explicit formula \eqref{ap14}, we obtain pointwise estimates \eqref{ap12}. 
The bounds \eqref{ap13} on $L^2$ and $L^\infty$-norms can be obtained in the same way as in \cite{LYZ} and we omit the details for brevity. The proof of Lemma \ref{lemA3} is completed.
\end{proof}

Finally, we have the following lemma concerning about the weight function $w(y)$ defined in \eqref{A1.1}. Set 
\begin{align}\label{ap8}
	\cw(y)= -(1-m^2U_s^2)U_s''-2m^2U_s|U_s'|^2.
\end{align}
\begin{lemma}\label{lemA2}
	Let the Mach number $m\in \left(0,\f{1}{\sqrt{3}}\right)$. For sufficiently small $|c|\ll 1$, the weight function $w$ has the following expansion
	\begin{align}
		w(y)=w_0(y)+cw_1(y)+O(1)|c|^2,\label{ap7}
	\end{align}
where 
\begin{align}\
	w_0&=(1-m^2U_s^2)^2\cw^{-1},\nonumber\\
	w_1&=4m^2U_s(1-m^2U_s^2)\cw^{-1}-2m^2(1-m^2U_s^2)^2\left(|U_s''|U_s+|U_s'|^2\right).\nonumber
\end{align}
Moreover, there exists a positive constant $\gamma_0$, such that
\begin{align}\label{ap10}
	w_0-U_sw_1\geq \gamma_0.
\end{align}
\begin{proof}
	We show that $\cw(y)$ has a strictly positive lower bound, hence $w_0$ and $w_1$ are well-defined. Then the expansion \eqref{ap7} follows from a straightforward computation. We take the plane Poiseuille flow $U_s(y)=1-y^2$ as an example.  Then
	\begin{align}
		\cw(y)=2(1-m^2)+6m^2y^2\left(y^2-\f23\right).\nonumber
	\end{align} 
Therefore,
$
	\min_{y\in [-1,1]} \cw(y)=\cw(y)|_{y=\pm \frac{1}{\sqrt{3}}}=2\left(1-\frac{4m^2}{3}\right),
$
which is positive when the Mach number $m\in \left(0,\frac{\sqrt{3}}{2}\right)$. For \eqref{ap10}, following the same computation as (3.29) in \cite{YZ}, one has $w_0-U_sw_1\sim 1-3m^2U_s^2$. Therefore, $w_0-U_sw_1$ has a strictly positive lower bound when the Mach number $m\in \left(0,\f1{\sqrt{3}}\right)$. The proof of Lemma \ref{lemA2} is completed.
\end{proof}
\end{lemma}

\bigbreak
\noindent{\bf Acknowledgment.} The research of Yang was supported by  the Hong Kong PhD Fellowship. The research of Zhang was  supported by the  Start-up Fund  (P0043862), and by the Research Centre for Nonlinear Analysis of  The Hong Kong Polytechnic University.

\end{document}